\documentclass[a4]{amsart}
\usepackage[margin=2.5cm]{geometry}

\def\3{\ss{}}
\usepackage{amsthm,amsmath,amssymb}
\usepackage{mathtools}
\usepackage{xcolor}

\usepackage{hyperref}

\newtheorem{theorem}{Theorem}[section]
\newtheorem*{theorem*}{Theorem}
\newtheorem*{corollary*}{Corollary}
\newtheorem{lemma}{Lemma}[section]
\newtheorem{proposition}{Proposition}[section]
\newtheorem{corollary}{Corollary}[section]
\newtheorem{conjecture}{Conjecture}[section]

\theoremstyle{definition}
\newtheorem{definition}{Definition}[section]
\theoremstyle{remark}
\newtheorem{remark}{Remark}[section]

\newtheorem{notation}{Notation}[section]

\DeclareMathOperator{\e}{\mathrm{e}}         
\newcommand{\tO}{\mathtt 0}                  
\newcommand{\tL}{\mathtt 1}                  

\newcommand{\rcp}[1]{\frac{1}{#1}}
\newcommand{\bl}{\begin{lemma}}
\newcommand{\el}{\end{lemma}}
\newcommand{\bp}{\begin{proposition}}
\newcommand{\ep}{\end{proposition}}
\newcommand{\bdf}{\begin{definition}}
\newcommand{\edf}{\end{definition}}
\newcommand{\bcor}{\begin{corollary}}
\newcommand{\ecor}{\end{corollary}}
\newcommand{\bpf}{\begin{pf}}
\newcommand{\epf}{\end{pf}}
\newcommand{\ZZ}{{\bf Z}}
\newcommand{\LandauO}{\mathcal O}
\newcommand{\eqdef}{\coloneqq}
\newcommand{\smallspace}{\hspace{0.75pt}} 
\newcommand{\Mid}{\mid\hspace{-0.8pt}\mid}   
\def\digitsum{\mathsf s}

\newcommand{\nO}{n_{\mathsf 0}}
\newcommand{\nL}{n_{\mathsf 1}}
\newcommand{\nLO}{n_{\mathsf {10}} }
\newcommand{\nLL}{n_{\mathsf {11}} }

\newcommand{\Mzero}{M_{\mathsf 0}}

\newcommand{\Mtwo}{M_{\mathsf 2}}
\newcommand{\Mthree}{M_{\mathsf 3}}
\newcommand{\Mfour}{M_{\mathsf 4}}
\newcommand{\Mfive}{M_{\mathsf 5}}
\newcommand{\Msix}{M_{\mathsf 6}}

\newcommand{\Ezero}{E_{\mathsf 0}}
\newcommand{\Eone}{E_{\mathsf 1}}
\newcommand{\Etwo}{E_{\mathsf 3}}
\newcommand{\Ethree}{E_{\mathsf 4}}
\newcommand{\Efour}{E_{\mathsf 2}}

\newcommand{\Kzero}{K_{\mathsf 0}}
\newcommand{\Kone}{K_{\mathsf 1}}

\newcommand{\Sone}{S_{\mathsf 1}}

\newcommand{\rzero}{r_{\mathsf 0}}

\newcommand{\factor}{\mathfrak m}

\newcommand{\biglbb}{\bigl[\hspace{-0.3em}\bigl[}
\newcommand{\bigrbb}{\bigr]\hspace{-0.3em}\bigr]}

\numberwithin{equation}{section}

\title[The joint distribution of binary and ternary digit sums]
{The joint distribution of binary and ternary digits sums}
\thanks{$^1$This research was supported by the Austrian Science Foundation 
FWF, grant P36137-N, by the FWF--ANR joint projects ArithRand (grant numbers I4945-N and ANR-20-CE91-0006) and SymDynAr (grant numbers 
I6750 and ANR-23-CE40-0024-01).
}

\author[Michael Drmota]{Michael Drmota${}^{*}$} 
\thanks{${}^{*}$Institute of Discrete Mathematics and Geometry, Technische 
Universit\"at Wien, Wiedner Hauptstra\3e 8-10/113, A-1040 Wien, Austria} 
\author[Lukas Spiegelhofer]{Lukas Spiegelhofer${}^{**}$} 
\thanks{${}^{**}$Department Mathematics and Information Technology, Montanuniversit\"at Leoben, Austria}

\date{\today}

\begin{document}

\begin{abstract}
We consider the sum-of-digits functions $s_2$ and $s_3$ in bases $2$ and $3$.
These functions just return the minimal numbers of powers of two (resp. three) needed in order to represent a nonnegative integer as their sum.
A result of the second author states that there are infinitely many \emph{collisions} of $s_2$ and $s_3$, that is, positive integers $n$ such that
\[s_2(n)=s_3(n).\]
This resolved a long-standing folklore conjecture.

In the present paper, we prove a strong generalization of this statement,
stating that $(s_2(n),s_3(n))$ attains almost all values in $\mathbb N^2$, in the sense of asymptotic density.
In particular, this yields \emph{generalized collisions}:
for any pair $(a,b)$ of positive integers, the equation
\[a\hspace{0.5pt}s_2(n)=b\hspace{0.5pt}s_3(n)\]
admits infinitely many solutions in $n$.
\end{abstract}

\maketitle

\def\({\left(}
\def\){\right)}

\section{Introduction and main result}\label{sec_introduction}

The number $\digitsum_q(n)$, for natural numbers $q\ge2$ and $n$, is the sum of the base-$q$ digits of $n$.
Since the base-$q$ expansion of $n$ can be found by the greedy algorithm, it is the \emph{lexicographically largest} representation of $n$ as sum of powers of $q$.
Using this, it is not difficult to show that $\digitsum_q(n)$ is the minimal number of powers of $q$ needed to represent $n$ as their sum:
\[
\digitsum_q(n)=\min\bigl\{k\ge0:\mbox{there exist }d_0,\ldots,d_{k-1}\in\mathbb N\mbox{ such that }n=q^{d_0}+\cdots+q^{d_{k-1}}\bigr\}.
\]
In the easiest case, the values of $\digitsum_2(n)$, as $n$ varies in $[2^\lambda,2^{\lambda+1})$, are distributed according to a binomial distribution with parameters $(1/2,\lambda)$.
It is not surprising that in general, the values of $\digitsum_q(n)$ are asymptotically normally distributed~\cite{DrmotaGajdosik1998,Katai1992}.
For example, the sum-of-digits function can be modeled by a sum of i.i.d. random variables on $\{0,\ldots,q-1\}$~\cite{Drmota2001}, from which the statement follows.

In the present paper, we consider the sum-of-digits function with respect to different bases $p$ and $q$ simultaneously.
The corresponding normal distributions concentrate around values that are many standard deviations apart~\cite{DeshouillersHabsiegerLaishramLandreau2017, Spiegelhofer2023collisions}.
Finding integers $n$ such that $\digitsum_2(n)-\digitsum_3(n)$ is small
may therefore be expected to be a non-trivial problem.

Towards the end of the last century, the first author received a hand-written letter from A.~Hildebrand, in which the following question was asked.
\begin{equation}\label{eqn_hildebrand_question}
\mbox{
Are there infinitely many positive integers $n$ such that $\digitsum_2(n)=\digitsum_3(n)$?
}
\end{equation}
A natural number $n$ such that $\digitsum_p(n)=\digitsum_q(n)$
will be called \emph{collision} of $\digitsum_p$ and $\digitsum_q$, or $(p,q)$-collision, in this paper.

The first author~\cite{Drmota2001} proved a statement on the joint distribution of $\digitsum_2(n)$ and $\digitsum_3(n)$,
using among others Baker's theorem on linear forms of logarithms.
A corollary of the main result (Corollary~2~\cite{Drmota2001} states the following.
\begin{corollary*}
Let $p,q>1$ be coprime integers. As $N\rightarrow\infty$, we have
\[\frac1N\#
\biggl\{
n<N:\frac{\digitsum_{p}(n)-(p-1)\log_{p}(N)/2}{\sqrt{(p^2-1)\log_{p}(N)/12}} <y_1, \, 
\frac{\digitsum_{q}(n)-(q-1)\log_{q}(N)/2}{\sqrt{(q^2-1)\log_{q}(N)/12}} <y_2
\biggr\}
\rightarrow
\Phi(y_1)\Phi(y_2).
\]
\end{corollary*}

The statement of a \emph{local version} of this result~\cite[Theorem~4]{Drmota2001} is at the core of our main theorem (Theorem~\ref{thm_main} below).
\begin{theorem*}[Drmota 2001]
Let $p,q>1$ be coprime integers, and $d=\gcd(p-1,q-1)$.
As $N\rightarrow\infty$, we have uniformly for all integers $k_1,k_2\ge0$ such that $k_1\equiv k_2\bmod d$,
\[
\begin{aligned}
\hspace{2em}&\hspace{-2em}
\frac1N\#\bigl\{n<N:\digitsum_{p}(n)=k_1,\digitsum_{q}(n)=k_2\bigr\}\\
&=
d 
\frac1{\sqrt{2\pi(p^2-1)\log_{p}(N)/12}}
\exp\biggl(-\frac{(k_\ell-(p-1)\log_{p}(N)/2)^2}
{2(p^2-1)\log_{p}(N)/12}\biggr) \\
&\times 
\frac1{\sqrt{2\pi(q^2-1)\log_{q}(N)/12}}
\exp\biggl(-\frac{(k_\ell-(q-1)\log_{q}(N)/2)^2}
{2(q^2-1)\log_{q}(N)/12}\biggr)
+o\bigl((\log N)^{-1}\bigr).
\end{aligned}
\]
\end{theorem*}
Concerning similar values of $\digitsum_2$ and $\digitsum_3$, 
Deshouillers, Habsieger, Laishram, and Landreau~\cite{DeshouillersHabsiegerLaishramLandreau2017} write
\begin{quote}
``[\ldots] it seems to be unknown whether there are infinitely many integers $n$ for which $\digitsum_2(n)=\digitsum_3(n)$ or even for which
$\lvert\digitsum_2(n)-\digitsum_3(n)\rvert$ is significantly small.''
\end{quote}
They prove the following theorem.
\begin{theorem*}
For sufficiently large $N$, we have
\[\#\bigl\{n\leq N:\lvert \digitsum_3(n)-\digitsum_2(n)\rvert\leq 0.1457205\log n\bigr\}>N^{0.970359}.\]
\end{theorem*}
The result is nontrivial since $\digitsum_3(n)-\digitsum_2(n)$ usually has a value around
$c\log n$, where
\[c=\frac 1{\log 3}-\frac 1{\log 4}=0.1888\ldots.\]
Thus, they obtain in fact infinitely many $n$ such that $\lvert \digitsum_2(n)-\digitsum_3(n)\rvert$ is ``significantly small''.

A partial refinement was given by 
La Bret\`eche, Stoll, and Tenenbaum~\cite{BretecheStollTenenbaum2019}.
They proved in particular that for all multiplicatively independent integers $p,q\ge2$, the set
\begin{equation}\label{eqn_BretecheStollTenenbaum2019}
\bigl\{\digitsum_p(n)/\digitsum_q(n):n\geq 1\}
\end{equation}
is dense in $\mathbb R^+$.

The article~\cite{DeshouillersHabsiegerLaishramLandreau2017} was the starting point of the paper~\cite{Spiegelhofer2023collisions} by the second author.
Applying a rarefaction by some power of three, we aligned the expected values of $\digitsum_2$ and $\digitsum_3$.
The existence of infinitely many $(2,3)$-collisions was then established by means of a suitable \emph{pre-selection of shifts} (see the three proof steps on page~482 of~\cite{Spiegelhofer2023collisions}).

\begin{theorem*}[Spiegelhofer 2023]
There are infinitely many positive integers $n$ such that $\digitsum_2(n)=\digitsum_3(n)$.
\end{theorem*}

The main theorem of the present paper yields this result, as well as the local theorem by Drmota stated above, and the special case $(p,q)=(2,3)$ of the La Bret\`eche--Stoll--Tenenbaum result as corollaries.

\begin{theorem}\label{thm_main}
Suppose that $0\le K \le c_2 \log N$, where $c_2 >0$. 
Then, we have uniformly for $K$ in this range and for all integers $k_1,k_2\ge 0$, 
as $N\to\infty$, 
\begin{eqnarray*}
&&\rcp{N}\#\left\{n<N\left|s_{2}(3^K n)= k_1, s_{3}(n)= k_2 \right.\right\}\\
&&\quad =  \frac 1{\sqrt{2\pi \frac{1}{4}\log_{2} (N\, 3^K) } }
\exp\left(- \frac{\left(k_1 - \frac{1}{2}\log_{2} (N\, 3^K) \right)^2}
{\frac{1}{2}\log_{2} (N\, 3^K) } \right) \\
&&\qquad \times  \frac 1{\sqrt{2\pi \frac{2}{3}\log_{3} N } }
\exp\left(- \frac{\left(k_2 - \log_{3} N \right)^2}
{\frac{4}{3}\log_{3} N } \right) \\
&& \quad + o\left( (\log N)^{-1} \right).
\end{eqnarray*}
An analogous statement holds with reversed roles of $2$ and $3$: in this case, rarefaction by $2^K$ with $s_3(2^K n)$ for even $n$ and $s_3(2^K n+1)$ for odd $n$ is used.
\end{theorem}

Combining the two parts of this theorem,
we obtain the following corollary.
\begin{theorem}\label{thm_local}
For every $\delta > 0$ there exists $K_0> 0$ such that for every 
pair of integers $(k_1,k_2)$ with
\[
k_1 \ge K_0,\ k_2 \ge K_0, \ k_1 \ge \delta k_2,\ k_2 \ge \delta k_1,
\]
there exists a non-negative integer $n$ satisfying
\[
\digitsum_2(n) = k_1 \quad\mbox{and}\quad \digitsum_3(n) = k_2.
\]
\end{theorem}

This result means that the pairs $(\digitsum_2(n),\digitsum_3(n))$ range over {\it almost all} possible pairs $(k_1,k_2)$.
Actually this result is best possible, in the light of the following result by Senge and Straus~\cite{SengeStraus1973}: for every pair $(k_1,k_2)$ of positive integers there are at most
finitely many non-negative integers $n$ such that $\digitsum_2(n) = k_1$ and $\digitsum_3(n) = k_2$.

\bigskip
Specializing further, Theorem~\ref{thm_local} immediately yields the following result.
\begin{corollary}\label{cor_special}
Assume that $a,b\ge1$ are integers. There exist infinitely many natural numbers $n$ such that
\begin{equation}\label{eqn_special}
a\smallspace\digitsum_2(n)=b\smallspace\digitsum_3(n).
\end{equation}
\end{corollary}

\subsection{Further directions}
\subsubsection{Catalan numbers, and $n!$ in base $12$.}

By Legendre's identity (valid for primes $p$) we have
\begin{equation}\label{eqn_legendre}
(p-1)\sum_{1\leq k\leq n}\nu_p(k)=n-\digitsum_p(n),
\end{equation}
and we see that the $p$-valuation of factorials, and hence combinatorial counting sequences formed by products of factorials, is intimately tied to the sum-of-digits function in base $p$.
For example, there is a direct connection between $(2,3)$-collisions and the base-$12$ expansion of $n!$~\cite{Deshouillers2012,Deshouillers2016,DeshouillersRuzsa2011}.
By~\eqref{eqn_legendre}, the integer $n\ge0$ is a collision if and only if
\[\nu_2(n!)=n-\digitsum_2(n)=2\smallspace\frac{n-\digitsum_3(n)}2=2\smallspace\nu_3(n!).\]
This is the case if and only if $n!$ is \emph{exactly divisible by} some power of $2^2\cdot 3^1=12$,
in symbols, $12^k\Mid n!$ for some $k$, where
\[12^k\Mid m\quad\Longleftrightarrow\quad 12^k\mid m\quad\textsf{and}\quad \gcd(12,m/12^k)=1.\]
In this case, and in this case only, the \emph{last significant base-$12$ digit} $\ell_{12}(n!)$ of $n!$ is an element of $\{1,5,7,11\}$.
Summarizing, we have the equivalences
\begin{equation}
\begin{array}{ll}
\digitsum_2(n)=\digitsum_3(n)
&\mbox{\textsf{if and only if}}\\[3mm]
\nu_2(n!)=2\smallspace\nu_3(n!)
&\mbox{\textsf{if and only if}}\\[3mm]
12^k\Mid n!\mbox{ for some }k
&\mbox{\textsf{if and only if}}\\[3mm]
\ell_{12}(n!)\in\{1,5,7,11\}.
\end{array}
\end{equation}
Together with J.-M.~Deshouillers and P.~Jelinek~\cite{DeshouillersJelinekSpiegelhofer2024}, second author proved that $\ell_{12}(n!)$ attains each digit in $\{1,\ldots,11\}$ infinitely many times, thus refining the theorem on the infinitude of $(2,3)$-collisions.

A related question concerns the $2$-and $3$-valuations of \emph{Catalan numbers}~\cite{DrmotaKrattenthaler2019,GaraevLucaShparlinski2006},
\[C_n=\frac1{n+1}\binom{2n}n.\]
\begin{conjecture}
Assume that $a,b\ge1$ are integers.
There exist infinitely many positive integers $n$ such that
\begin{equation}\label{eqn_catalan_collisions}
a\smallspace\nu_2\bigl(C_n\bigr)=b\smallspace\nu_3\bigl(C_n\bigr).
\end{equation}
\end{conjecture}
If $\gcd(a,b)=1$, this states that $C_n$ is exactly divisible by some power of $2^b3^a$ infinitely often.

More generally, in the spirit of our main theorem (Theorem~\ref{thm_main}), we could ask whether $(\nu_2(C_n),\nu_3(C_n))$ attains all values in the set
\[\bigl\{(k_1,k_2)\in\mathbb N^2: \lvert k_1+ik_2\rvert>K, \varepsilon<\arg(k_1+ik_2)<\pi/2-\varepsilon\bigr\},\]
where $\varepsilon>0$, and $K=K(\varepsilon)$ is large enough.
We leave this as another open problem.

\subsubsection\textbf{Collisions in different bases.}
P.~Jelinek (private communication) 
announced a proof of the existence of infinitely many collisions with respect to any pair $(p,q)$ of coprime bases $p,q\ge2$.
As a possible extension, we could again ask for corresponding statements concerning the prime factorization of Catalan numbers.

Collisions in more than two bases --- $\digitsum_p(n)=\digitsum_q(n)=\digitsum_r(n)$ --- are in general certainly very difficult to handle.
For example, it follows from work in progress by Jelinek that there exist infinitely many such collisions for some triples of bases, but it is not so clear what happens in the general case.
If two of the three bases are much larger than the third, we would need to reduce significantly the sum of digits in two bases synchronously in order to obtain a collision. Currently we do not see a way to achieve this.
In this context, it might be of interest to recall an ergodic conjecture~\cite{Furstenberg1970} by Furstenberg, concerning \emph{multiplicatively independent} integer bases $p,q\geq 2$:
let $\dim_H(A)$ be the Hausdorff dimension of a set $A\subseteq[0,1]$,
 and define
\[O_a(x)\coloneqq \bigl\{a^kx \bmod 1:k\in\mathbb N\bigr\}\]
Then
\begin{equation}\label{eqn_furstenberg}
\dim_H\bigl(\overline{O_p(x)}\bigr)+\dim_H\bigl(\overline{O_q(x)}\bigr)\geq 1
\end{equation}
for all irrational $x\in[0,1]$.
In other words, the base-$p$ and base-$q$ expansions of an irrational number cannot be ``simple'' at the same time (see Shmerkin~\cite{Shmerkin2019} and Wu~\cite{Wu2019} for partial solutions of this conjecture, and Adamczewski--Faverjon~\cite{AdamczewskiFaverjon2020} for solutions of several problems concerning the joint representation of a number in two bases).

\subsection{Auxiliary results}

Theorem~\ref{thm_main} (and consequently Theorem~\ref{thm_local}) follows directly 
from the following two propositions (see Section~\ref{sec:2}).

\begin{proposition}\label{prp_1}
Suppose that $c_2>0$ is a real number.
There exists $c>0$ such that uniformly for all integers $K\ge0$ satisfying
$0\le K \le c_2 \log N$, and all real $t_1$ and $t_2$,
\[
S_1^{(i)}  = \sum_{n< N, \, n \equiv i \bmod 2} e\left( t_1 \digitsum_2(3^K n) + t_2 \digitsum_3(n) \right) \ll N\, \exp\left( - c \log N 
\| t_1\|^2 
\right)
\]
and
\[
S_2^{(i)} = \sum_{n< N, \, n \equiv i \bmod 2} e\left( t_1 \digitsum_2(n) + t_2 \digitsum_3(2^Kn+r) \right) \ll N\, \exp\left( - c \log N 
\|2t_2\|^2 
\right).
\]
for $i,r\in \{0,1\}$.
\end{proposition}

\begin{remark}
We will prove in detail only the first of these formulas, while the proof of the second is analogous.
These two statements correspond to the adjusting the expected values of $\digitsum_2$ and $\digitsum_3$ in opposite directions, thus allowing for $\digitsum_2(n)/\digitsum_3(n)$ to be ``large'' and ``small'', respectively.
\end{remark}

\begin{proposition}\label{Pro2}
Suppose that $c_2>0$ is a real number. Then we have uniformly for all integers $K\ge0$ satisfying
$0\le K \le c_2 \log N$, and all real $t_1$ and $t_2$
\begin{align*}
S_1^{(i)} &= \sum_{n< N\, n \equiv i \bmod 2} e\left( t_1 s_2(3^K n) + t_2 s_3(n) \right) = 
\frac N2\, e\left(\frac{t_1}2 \log_2(3^K N) + t_2 \log_3(N) \right) \\
&\times \exp \left( 
- \frac {\pi^2}2  t_1^2\log_2(3^K N) - \frac{4\pi^2}3 t_2^2\log_3(N) 
\right) + o(N)
\end{align*}
and 
\begin{align*}
S_2^{(i)} &= \sum_{n< N\, n \equiv i \bmod 2} e\left( t_1 s_2(n) + t_2 s_3(2^Kn) \right) = 
\frac N2\, e\left( \frac{t_1}2 \log_2(N) + t_2 \log_3(2^K N) \right) \\
&\times \exp \left( 
- \frac{\pi^2 }2 t_1^2\log_2(N) - \frac{4 \pi^2}3 t_2^2\log_3(2^K N+r) 
\right) + o(N)
\end{align*}
for $i,r\in \{0,1\}$.
\end{proposition}

\subsection{Plan of the paper.}

We will prove first that Propositions~\ref{prp_1} and~\ref{Pro2} imply the main theorems. 
In a short section we collect some Diophantine properties that will be then used
in the  subsequent two sections that are concerned with the proofs of Propositions~\ref{prp_1} and~\ref{Pro2}.
\begin{notation}
The symbol $\log$ denotes the natural logarithm,
and $\log_a=\frac 1{\log a}\log$ is the logarithm in base $a>1$.
We use Landau notation, employing the symbols $\LandauO$, $\ll$, and $o$.
The symbol $f(n)\asymp g(n)$ abbreviates the statement
$\bigl(f(n)\ll g(n)$ \textsf{and} $g(n)\ll f(n)\bigr)$,
while $f(n)\sim g(n)$ means that $f(n)/g(n)$ converges to $1$ as $n\rightarrow\infty$.
We also use the exponential $\e(x)=\exp(2\pi i x)$.
For $M\geq 0$, the statement ``$a$ is $M$-close to $b$'' means $\lvert a-b\rvert\leq M$.
\end{notation}

\section{Propositions~\ref{prp_1} and \ref{Pro2} imply Theorems~\ref{thm_main} and~\ref{thm_local}}\label{sec:2}.

We set $S_1^{(i)} = S_1^{(i)}(t_1,t_2)$ as in Propositions~\ref{prp_1} and \ref{Pro2}. Then we have
\[
\#\{ n < N : s_2(3^K n) = k_1, \ s_3(n) = k_2 \} 
= \iint_{[-1/2,1/2]^2}  \left( S_1^{(0)}(t_1,t_2) + S_1^{(1)}(t_1,t_2) \right) 
e(-t_1 k_1 - t_2 k_2) \, dt_1\, dt_2.
\]
Futhermore we set
\[
C_L^{(0)} = \left[ - \frac {L}{\sqrt{\log N}} , \frac {L}{\sqrt{\log N}}  \right]^2, \quad
C_L^{(1)} =   
\left[ - \frac {L}{\sqrt{\log N}} , \frac {L}{\sqrt{\log N}}  \right] \times
\left(  \left[ - \frac 12 , -\frac 12 + \frac {L}{\sqrt{\log N}}  \right] 
\cup  \left[ \frac 12 - \frac {L}{\sqrt{\log N}} , \frac 12  \right]   \right) 
\]
and
\[
A_L = \left[ - \frac 12, \frac 12 \right]^2 \setminus (C_L^{(0)} \cup C_L^{(1)}).
\]
By Proposition~\ref{prp_1} it directly follows that.
\[
I_1 = \iint_{A_L} |S_1^{(i)}(t_1,t_2)|\, dt_1\, dt_2 \ll N \frac{e^{-c L^2}}{\log N}.
\]

Next we apply Propositions~\ref{Pro2} and observe that for every $\varepsilon > 0$ there
exists $N_0 = N_0(\varepsilon)$ such that 
\[
\left| S_1^{(i)}(t_1,t_2) -  \frac N2\,  e\left(\frac{t_1}2 \log_2(3^K N) + t_2 \log_3(N) \right) 
\exp\left( -\frac {\pi^2}2  t_1^2\log_2(3^K N) - \frac{4\pi^2}3 t_2^2\log_3(N) 
\right)\right| \le \varepsilon N
\] 
for all $N\ge N_0$ and (uniformly) for all real $t_1,t_2$ and $i\in\{0,1\}$.
In order to calculate the integral 
\[
I_2 = \iint_{C_L^{(0)} \cup C_L^{(1)}} \left(S_1^{(0)}(t_1,t_2)+ S_1^{(1)}(t_1,t_2)\right) 
e(-t_1 k_1 - t_2 k_2)\, dt_1\, dt_2 
\]
we observe that (due to the fact that $s_3(n) \equiv n \bmod 2$) we have the relation
$S_1^{(i)}(t_1,t_2+1/2) = (-1)^i S_1^{(i)}(t_1,t_2)$ and consequently
\begin{align*}
\iint_{C_L^{(1)}} S_1^{(i)}(t_1,t_2) 
e(-t_1 k_1 - t_2 k_2)\, dt_1\, dt_2 
&= \iint_{C_L^{(0)}} S_1^{(i)}(t_1,t_2+1/2) 
e(-t_1 k_1 - (t_2+1/2) k_2)\, dt_1\, dt_2 \\
&= (-1){i+k_2} \iint_{C_L^{(0)}} S_1^{(i)}(t_1,t_2) 
e(-t_1 k_1 - t_2 k_2)\, dt_1\, dt_2
\end{align*}
Thus
\[
I_2 = 2 \iint_{C_L^{(0)}} S_1^{( k_2 \bmod 2)}(t_1,t_2) 
e(-t_1 k_1 - t_2 k_2)\, dt_1\, dt_2.
\]
Next we use the simple formula
\[
\int_{|t|\le C} e^{iAt- \frac{t^2}2 B}\, dt = 
\sqrt{\frac{2\pi}{B}} e^{-\frac{A^2}{2B}} + O\left( \frac 1{BC} e^{-\frac{C^2 B}2} \right)
\]
and Propostion~\ref{Pro2} and obtain
\begin{align*}
I_2 & = \frac N{\sqrt{\frac 23 \pi \log_2(3^KN) \log_3( N }}
\exp\left( - \frac{ 2 \Delta_1^2 }{\log_2(3^KN) } - \frac{ 3 \Delta_2^2}{ 4 \log_3 (N) } \right) \\
&+ O\left(  \frac N{L \log N} e^{ - c L^2 }   \right) + O\left( N \frac{L^2 \varepsilon}{\log N} \right)
\end{align*}
for some constant $c>0$, where
\[
\Delta_1 = k_1 - \frac 12 \log_2(3^K N) \quad\mbox{and}\quad
\Delta_2 = k_2 - \log_3(N).
\]
Finally we can choose 
\[
L = \lfloor  \sqrt{ (1/c) \log (1/\varepsilon) } \rfloor
\]
so that the error term sums up to 
\[
O\left( \frac{N}{\log N} \varepsilon \log (1/\varepsilon)   \right)
\]
for $N\ge N_0(\varepsilon)$.
This proves the first part of Theorem~\ref{thm_main}.

The proof of the second part is very similar. The only difference is that
we use $s_3(2^K n)$ to cover even $k_2$ and $s_3(2^K n+1)$ to cover odd $k_2$.

\bigskip

For the proof of Theorem+\ref{thm_local} we fix $\delta > 0$  and set 
\[
c_2 = \frac 1{\log 3}\left( \frac {2\log 2}{\delta \log 3} - 1  \right).
\]
With this parameter we apply Theorem~\ref{thm_main} and now choose $N_0$ large enough such
that the error term $o( (\log N)^{-1})$ in Theorem~\ref{thm_main} is negligible for
all $N\ge N_0$ compared to the main term
\[
T:= \frac N{\sqrt{2\pi \frac{1}{4}\log_{2} (N\, 3^K) } }
\exp\left(- \frac{\Delta_1^2} 
{\frac{1}{2}\log_{2} (N\, 3^K) } \right) 
\frac 1{\sqrt{2\pi \frac{2}{3}\log_{3} N } }
\exp\left(- \frac{\Delta_2^2}
{\frac{4}{3}\log_{3} N } \right)
\]
when $|\Delta_1| \le 1 + \frac {\log 3}{2\log 2}$ and $|\Delta_2| \le 1$. 
Furthermore we can assume that the main term $T$ is greater than $1$.

Now assume that $(k_1,k_2)$ is a pair of positive integers satisfying
\[
\frac{\log 3}{2 \log 2} \le \frac {k_1}{k_2} \le \frac 1\delta
\]
and 
\[
\max\{ k_1,k_2\} \ge K_0 := \frac{\log N_0}{\delta \log 3} + 1.
\]
We then choose $N\ge N_0$ and $0\le K \le c_2 \log N$ such that
\[
|\Delta_1| = |k_1 - \log_3 N| \le 1 \quad \mbox{and} \quad
|\Delta_2| = \left|k_2 - \frac 12 \log_2(3^K N)\right| \le 1 + \frac {\log 3}{2\log 2}.
\]
With this choice Theorem~\ref{thm_main} shows that 
\[
\# \{ n < N : s_2(3^K n) = k_1, \ s_3(n) = k_2 \} = T(1+o(1)) > 1.
\]
This proves Theorem~\ref{thm_local} 
in the case $k_1/k_2 \ge \log 3/(2 \log 2)$. The other case runs
along the same lines.

\section{Diophantine Properties}

The first property (Lemma~\ref{CorX}, compare also with
\cite{Drmota2001,DrmotaMauduitRivat2019twobases}) follows from Baker's theorem on linear forms of logarithms
(see \cite{Waldschmidt1993}).

\begin{lemma}\label{LeBaker}
Let $\alpha_1,\alpha_2,\ldots, \alpha_n$ be non-zero algebraic numbers
and $b_1,b_2,\ldots, b_n$ integers such that 
\[
\alpha_1^{b_1}\cdots \alpha_n^{b_n} \ne 1
\]
and let $A_1,A_2,\ldots,A_n\ge e$ real numbers with $\log A_j\ge h(\alpha_j)$,
where $h(\cdot)$ denotes the absolute logarithmic height. Set
$d = [{\bf Q}(\alpha_1\ldots,\alpha_n):{\bf Q}]$. Then
\[
\left|\alpha_1^{b_1}\cdots \alpha_n^{b_n} - 1 \right| \ge 
\exp\left( - U\right),
\]
where 
\[
U = 2^{6n+32} n^{3n+6} d^{n+2}(1 + \log d)(\log B + \log d)\log A_1\cdots \log A_n
\]
and
\[
B = \max\{2,|b_1|,|b_2|,\ldots,|b_n|\}.
\]
\end{lemma}

\begin{lemma}\label{CorX}
Let $q_1,q_2>1$ be coprime integers and $m_1,m_2$ integers such that
$m_1\not\equiv 0 \bmod q_1$ and $m_2\not\equiv 0 \bmod q_2$. Then
there exists a constant $C>0$ such that for all positive integers $k_0,k_1,k_2>1$
\[
\left| \frac{m_1 q_2^{k_0}}{q_1^{k_1}}+\frac{m_2}{q_2^{k_2}}\right | 
\ge \max\left( \frac{|m_1|q_2^{k_0}}{q_1^{k_1}},\frac{|m_2|}{q_2^{k_2}} \right)\cdot
e^{- C \log q_1 \log q_2 \log \left(\max(k_1,k_0+k_2)\right)\cdot 
\log(\max\left(|m_1|,|m_2|\right))}.
\]
\end{lemma}

\begin{proof}
Since $q_1,q_2>1$ are coprime integers and
$m_1\not\equiv 0 \bmod q_1$, $m_2\not\equiv 0 \bmod q_2$ we surely have
$m_1 q_1^{-k_1} q_2^{k_0} + m_2 q_2^{-k_2}\ne 0$. So can apply
Lemma~\ref{LeBaker} for $n=3$, $\alpha_1 = q_1$, $\alpha_2=q_2$,
$\alpha_3 = -m_2/m_1$, $b_1 = k_1$, $b_2 = -k_2$, $b_3 = 1$ and directly obtain
\begin{align*}
\left| \frac{m_1q_2^{k_0}}{q_1^{k_1}}+\frac{m_2}{q_2^{k_2}}\right| &=
|m_1|\cdot q_1^{-k_1}q_2^{k_0}\cdot\left|-q_1^{k_1}q_2^{-k_0-k_2}\frac {m_2}{m_1} - 1\right| \\
&\ge |m_1|q_1^{-k_1}q_2^{k_0}e^{- C \log q_1 \log q_2 \log \left(\max(k_1,k_0+k_2)\right)\cdot 
\log\max\left(|m_1|,|m_2|\right)}.
\end{align*}
In the same way we get the lower bound
\[
\left| \frac{m_1q_2^{k_0}}{q_1^{k_1}}+\frac{m_2}{q_2^{k_2}}\right| 
\ge |m_2|q_2^{-k_2}e^{- C \log q_1 \log q_2 \log \left(\max(k_1,k_0+k_2)\right)\cdot 
\log\max\left(|m_1|,|m_2|\right)}
\]
which completes the proof of the lemma. 
\end{proof}

The second one follows from the 
$p$-adic version of the subspace theorem by Schlickewei \cite[Theorem~1.8]{Evertse1993}.

\begin{lemma}\label{LeDio2}
Let $r \ge n \ge 2$, $C > 0$, $\delta > 0$ and $S = \{\infty,p_1,\ldots,p_t\}$, 
where $p_1,\ldots,p_t$ are distinct prime numbers. Further, 
let $L_{1,\infty},\ldots,L_{r,\infty}$ be linear forms in $X_1,\ldots,X_n$ with algebraic coefficients 
in $\mathbb{C}$ in general position, and for $1\le j\le t$, let $L_{1,p_j},\ldots,L_{r,p_j}$ 
be linear forms in $X_1,\ldots,X_n$ with algebraic coefficients in $\overline {\mathbb{Q}}_{p_j}$ in general position. 

Then all integer solutions ${\bf x} = (x_1,\ldots,x_n)$ with ${\rm gcd}(x_1,\ldots,x_n) = 1$ of the 
inequality
\begin{equation}\label{eqLeDio2}
\prod_{p\in S} \left| L_{1,p}({\bf x}) \cdots L_{r,p}({\bf x}) \right|_p \le C \| {\bf x} \|_\infty^{r-n-\delta}
\end{equation}
are contained in the union of finitely many linear subspaces of ${\mathbb{Q}}^n$.
\end{lemma}

The following two properties are corollaries of Lemma~\ref{LeDio2}.

\begin{lemma}\label{LeDio3}
Suppose that $q_1,q_2$ are different prime numbers, that $h_1,\ldots, h_d$ are $d\ge 1$ integers not
divisible by $q_2$ and $H$ an integer not divisible by $q_1$ such that
\[
{\rm gcd}(h_1,\ldots, h_d, H) = 1.
\]
Then for every $\delta > 0$ there exists $M_0$ such that we have uniformly for all integer exponents 
$k$, $m_1, \ldots, m_d$, $m$
with $k\ge 0$, $m_1 > m_2 > \cdots > m_d = 0$ and  $m \ge \max(m_1, M_0)$ the inequality
\begin{equation}\label{eqLeDio3}
\left| q_1^{k} \left( q_2^{m_1} h_1 + \cdots +  q_2^{m_d} h_d \right) - q_2^m H  \right|
\gg \frac{ \max ( q_1^{k} q_2^{m_1}, q_2^m |H| )^{1-\delta} }{|h_1h_2 \cdots h_d\, H| },
\end{equation}
where the implicit constant depends just on $\delta$ and $d$.
\end{lemma}

\begin{proof}
We apply Lemma~\ref{LeDio2} with $r=d+2$, $n=d+1$, $C=1$, $\delta$, and the set $S$ that consists of $\infty$ and of
the primes $q_1$ and $q_2$. For $p\in S$ we set 
\begin{displaymath}
  L_{j,p} = X_j,\ (1\le j\le d+1),\ 
  L_{d+2,p} = X_1+X_2+\cdots + X_{d+1}, 
\end{displaymath}
that are obviously in general position. Hence for all coprime integer
tuples ${\bf x} = (x_1,x_2,\ldots x_{d+1})$ we either have
\begin{equation}\label{eq:linear-form-upperbound}
\prod_{p\in S} \left| (x_1+x_2+\cdots + x_{d+1}) x_1x_2 \cdots x_{d+1} \right|_p 
\ge \| {\bf x} \|_\infty^{1-\delta}
\end{equation}
or they are contained in finitely many linear subspaces of ${\mathbb{Q}}^{d+1}$.

We now set
\[
x_j = q_1^{k} q_2^{m_j} h_j \ (1\le j \le d), \ x_{d+1} = q_2^m H.
\]
Clearly we have
\[
\| {\bf x} \|_\infty \ge \max\left(  q_1^{k} q_2^{m_1}, q_2^m |H| \right). 
\]
By assumption $m_d = 0$ and $h_d$ is not divisible by $q_2$. Furthermore $H$ is not
divisible by $q_1$. Consequently we have
\[
{\rm gcd} \left( q_1^{k} q_2^{m_1} h_1, \ldots, q_1^{k} q_2^{m_d} h_d, q_2^m H  \right)
= {\rm gcd} \left(  h_1, \ldots,  h_d,  H  \right) = 1.
\]

Suppose now that $c_1 x_1 + \cdots c_d x_d + c_{d+1} x_{d+1} = 0$ is one (of finitely many) equations
for the exceptional rational subspaces, that is, we can assume that the coefficients $c_j$ are
integers and not all of them are zero. Actually, since all $x_j$ are non-zero we can assume that
at least two coefficients $c_j$ are non-zero. In particular this implies that at least one of
the coefficients $c_1,\ldots, c_d$ is non-zero. 

Suppose first that $d=1$. Then $c_1 \ne 0$ and $m_1 = 0$ and by considering the equation 
\[
c_1 q_1^{k}h_1  + c_2 q_2^m H = 0
\]
modulo $q_2^{m}$ we get
\[
c_1 \equiv 0 \bmod q_2^m.
\]
If $M_0$ is chosen large enough  such that this relation is impossible for $m\ge M_0$ 
then there are no points of this form on one of the finitely many subspaces.

Next assume that $d> 1$ and that $c_d \ne 0$. Here we consider the
equation 
\[
c_1 q_1^{k} q_2^{m_1} h_1 + \cdots + c_d  q_1^{k} q_2^{m_d} h_d + c_{d+1} q_2^m H = 0
\]
modulo $q_2^{m_{d-1}}$. Since $m \ge m_1 > \cdots > m_{d-1} > 0$, $h_d$ is not divisibly by $q_2$, and $q_1$ and $q_2$ are different
prime numbers it follows that
\[
c_d \equiv 0 \bmod q_2^{m_{d-1}}. 
\]
If $m_{d-1}$ is sufficiently large this is certainly impossible. 

Similarly as above, if $c_d=0$ but $c_{d-1} \ne 0$ we get
\[
c_{d-1} q_2^{m_{d-1}} \equiv 0 \bmod q_2^{m_{d-2}} \quad\mbox{or} \quad c_{d-1} \equiv 0 \bmod q_2^{m_{d-2}-m_{d-1}}.
\]
Again this is impossible if $m_{d-2}-m_{d-1}$ is sufficiently large. 

In this way we proceed further and observe that there exists a constant $C > 0$ (depending on $\delta$)
such that $(x_1,\ldots,x_d, x_{d+1})$ is not contained in any of the exceptional subspaces provided that
$m_{d-1} \ge C$, $m_{d-2}-m_{d-1} \ge C$, $\ldots$, and $m_{2}- m_1 \ge C$. 
In all these cases the inequality (\ref{eq:linear-form-upperbound}) is satisfied.

By definition we have
\begin{displaymath}
  \prod_{p\in S}
  | q_1^{\ell_1 + \cdots + \ell_d} q_2^{m_1+ \cdots m_d + m} |_p   = 1
\end{displaymath}
so that
\begin{displaymath}
  \prod_{p\in S} |x_1 \cdots x_{d+1} |_p
  = \prod_{p\in S} | h_1 \cdots h_d\, H |_p
  \le |h_1 \cdots h_d\, H | .
\end{displaymath}
Furthermore with $|x_1+\cdots + x_{d+1}|_p \le 1$
for $p \ne \infty$ it 
follows that \eqref{eq:linear-form-upperbound} implies
\begin{align*}
|x_1+\cdots + x_{d+1}|\, |h_1 \cdots h_d\, H | &\ge   \prod_{p\in S\setminus\{\infty\}}
 |x_1+ \cdots x_{d+1}|_p |x_1|_p \cdots |x_{d+1} |_p \\
& \ge   \max (|x_1|, \ldots |x_{d+1}|)^{1-\delta} \\
&\ge \max \left( q_1^{k} q_2^{m_1}, q_2^m |H|  \right)^{1-\delta} 
\end{align*}
which is precisely (\ref{eqLeDio3}).

Now suppose that $m_{d-1}\le C$. Then we put the terms
\[
q_1^{k} q_2^{m_{d-1}} h_{d-1} + q_1^{k} h_d = q_1^k \left( q_2^{m_{d-1}} h_{d-1} +  h_d    \right) 
\]
together and apply inductively the lemma for the case $d-1$. Since $q_2^{m_{d-1}}\le q_2^C$ is bounded
we have
\[
\left| q_2^{m_{d-1}} h_{d-1} +  h_d \right| \ll |h_{d-1} h_d|.
\]
Thus, (\ref{eqLeDio3}) follows by induction.

Similarly, if $m_{d-2}-m_{d-1} \ge C$ then we group together the terms
\[
q_1^{k} q_2^{m_{d-2}} h_{d-2} + q_1^{k} q_2^{m_{d-1}} h_{d-1} = q_1^k   q_2^{m_{d-1}}
\left( q_2^{m_{d-2}- m_{d-1}} h_{d-2} +  h_{d-1}    \right) 
\]
and proceed in the same way. The remaining cases can be handled, too.
However, we have to take care of the maximum $\max ( q_1^{k} q_2^{m_1}, q_2^m |H| )$
if we group together
\[
q_1^k q_2^{m_1} h_1 + q_1^k q_2^{m_2} h_2 = q_1^k q_2^{m_2} \left( q_2^{m_1-m_2} h_1 + h_2 \right).
\]
Since $m_1-m_2 \le C$ we get
\[
q_1^k q_2^{m_2} \ge \frac 1{q_2^C} q_1^k q_2^{m_1}.
\]
Thus, we can proceed by induction in all cases.

Summing up, we either get (\ref{eqLeDio3}) directly or we reduce it to the case $d-1$.
Since the case $d=1$ always holds (provided that $m$ is sufficiently large) the proof
of the lemma is finished.
\end{proof}

\begin{lemma}\label{LeDio4}
Suppose that $q_1,q_2$ are different prime numbers, that $h_1,\ldots, h_{d_1}$ are $d_1\ge 1$ integers not
divisible by $q_2$, $r_1,\ldots, r_{d_2}$ are $d_2\ge 1$ integers not
divisible by $q_1$ and $H$ an integer not divisible by $q_1q_2$  such that
\[
{\rm gcd}(h_1,\ldots, h_{d_1}, r_1, \ldots, r_{d_2}, H) = 1.
\]
Then for every $\delta > 0$ there exists $M_0$ such that we have uniformly for all integer exponents 
$k$, $m_1, \ldots, m_{d_1}$, $M$,  $n_1, \ldots, n_{d_2}$, $N$
with $k\ge 0$, $m_1 > m_2 > \cdots > m_{d_1} = 0$, $M \ge \max(m_1, M_0)$,  
$n_1 > n_2 > \cdots > n_{d_2} = 0$, and  $N \ge \max(n_1, M_0)$ the inequality
\begin{align}
&\left| q_1^{k} \left( q_2^{m_1} h_1 + \cdots +  q_2^{m_{d_1}}  h_{d_1} \right) +  
q_1^{n_1} r_1 + \cdots +  q_1^{n_{d_2}}  r_{d_2}  - q_1^N q_2^M H  \right| \nonumber \\
& \gg \frac{ \max ( q_1^{k} q_2^{m_1}, q_1^{n_1}, q_1^N q_2^M |H| )^{1-\delta} }
{|h_1h_2 \cdots h_{d_1}r_1r_2 \cdots r_{d_2}\, H| },  \label{eqLeDio4}
\end{align}
where the implicit constant depends just on $\delta$, $d_1$, and $d_2$.
\end{lemma}

\begin{proof}
The proof of Lemma~\ref{LeDio4} is a direct extension of the proof of Lemma~\ref{LeDio3}.
\end{proof}

%

\section{Proof of Proposition~\ref{prp_1}}
For an integer $M\ge1$, let $L(M)$ denote the length of the longest block of $0$s or $1$s in the binary expansion of $M$.
The following lemma states that almost all powers of $3$ only have short runs of $0$s or $1$s, where also multiplication by a factor $M$ is taken into account.
This inconspicious lemma is in fact the key to the proof of Proposition~\ref{prp_1}, as it enables us to eliminate binary digits of powers of $3$ with indices lying in an interval (Corollary~\ref{cor_odd} below).
\begin{lemma}\label{lem_not_too_many_digits}
Assume that $0\leq\eta\leq1$.
Then
\begin{equation}\label{eqn_small_o}
\sup_{1\leq M<2^{\eta K}}
L\bigl(M3^K\bigr)
\leq \eta K+o(K)
\end{equation}
as $K\rightarrow\infty$.
In particular,
\begin{itemize}
\item[(1.)]
the longest $0$-or $1$-blocks in the binary expansion of $3^K$ have length $o(K)$ as $K\to\infty$.
\item[(2.)] For given $\varepsilon>0$ and $\eta>0$, all sufficiently large $K$ satisfy
\begin{equation*}
\sup_{1\leq M<2^{\eta K}}
L\bigl(M3^K\bigr)
\leq (1+\varepsilon)\eta K.
\end{equation*}
\end{itemize}
\end{lemma}

\begin{proof}
The proof is an application of Schlickewei's $p$-adic subspace theorem.
Suppose that the binary expansion of $M3^K$ has a $0$-block of length $L$.
Then $M3^K$ can be represented as 
\[
M3^K = a + 2^{k+L}b,
\]
where $0 <  a\le 2^k$ and $0<b\le M3^K 2^{-k-L}$. Hence, 
\[
|M3^K - 2^{k+L}b| \le 2^k.
\]
On the other hand we have (by a direct application of the $p$-adic subspace theorem, see below)
\begin{equation}\label{eqlowerbound1}
|M3^K - 2^{k+L}b| \ge \frac{\max(3^K,2^{k+L}b)^{1-\delta}}{Mb} \ge \frac{(2^{k+L}b)^{1-\delta}}{Mb}
\end{equation}
or we have 
\[
c_1 M3^K + c_2 2^{k+L}b = 0
\]
for one (or several) of finitely many integer pairs $(c_1,c_2)\ne (0,0)$ that depend on $\delta$. 
However, by considering such an equation modulo $2^{k+L}$ 
it follows that
\[
c_1 \equiv 0 \bmod 2^{k+L}
\]
which is impossible if $k+L$ is sufficiently large. 

Thus, if $k+L$ is sufficiently large we certainly
have (\ref{eqlowerbound1}). Consequently we have
\[
2^k \ge \frac{(2^{k+L}b)^{1-\delta}}{Mb}
\]
or
\[
2^L \le M^{\frac1{1-\delta}}(2^k b)^{ \frac {\delta}{1-\delta} }.
\]
Since $2^k b \le M3^K$ and $M\leq 2^{K\eta}$ it also follows that
\[
L \le \biggl(\frac{\eta}{1-\delta}+
\frac{(\eta\log2+\log3)\delta}{1-\delta}
\biggr)K.
\]
Since $\delta > 0$ can be chosen arbitrarily small it follows that $L\leq 2\eta K+o(K)$, as proposed.
\end{proof}


From Lemma~\ref{lem_not_too_many_digits} we derive the following corollary, similar in spirit to the ``odd elimination lemma'' in the manuscript~\cite{Spiegelhofer2023cubes} by the second author.
To this end, let us introduce the convenient notation
\begin{equation}\label{eqn_block_extraction_def}
n^I\eqdef\sum_{a\leq j<b}\delta_j(n)2^{j-a},
\end{equation}
for $n\in\mathbb N$ and an interval $I=[a,b)$ in $\mathbb N$,
where $\delta_j(n)$ is the base-$2$ digit of $n$ at index $n$.
\begin{corollary}[Odd elimination]\label{cor_odd}
Assume that $\varepsilon,\eta>0$,
and assume that $d$ and $k$ are positive integers.
For all $K\ge \Kzero(\eta,\varepsilon,k)$, 
the following statement holds.
\begin{quote}
For all nonnegative integers $a,b$ such that
\[
\varepsilon K\leq a<b\leq a+\eta K,
\]
and all $\omega\in\{0,\ldots,2^{b-a}-1\}$,
there exists $A\in d+k\mathbb N$ such that $0<A\leq 4k^22^{2\eta K+\varepsilon K}$
and 
\[
\bigl(A3^K\bigr)^{[a,b)}
=\omega.
\]
\end{quote}

\end{corollary}
The idea of proof of this statement is the following.
First, we choose a factor $M\ge1$, by Dirichlet's approximation theorem, such that $M3^K$ does not have binary digits in the interval $[a-m,b)$, where $m$ is a small margin coming from the modulus $k$.
Due to Lemma~\ref{lem_not_too_many_digits} it is not possible that ``too many'' digits below $a-m$ are eliminated by such a multiplication.
Therefore there exist binary digits equal to $0$ and $1$ not too far below the cleared interval of digits of $M3^K$.
This will enable us to find a factor $A$ in a prescribed residue class: we will r
require $2\nmid A$ in order to allow for uniform distribution of the lowest digits in base $2$ (see~\eqref{eqn_before_CLT} for details).

\begin{proof}[Proof of Corollary~\ref{cor_odd}]
Choose $m\ge1$ in such a way that $2^{m-1}\leq k<2^m$.
Let $\eta\ge0$, and set $\kappa\eqdef b-a+m$. 
By Dirichlet's approximation theorem we may choose an integer $C=C(K)\in\{1,\ldots,2^\kappa\}$ in such a way that
\[\lVert C3^K2^{-b}\rVert<2^{-\kappa}.\]
That is, the digits of $C3^K$ with indices in the interval $[a-m,b)$ are all equal to $\tO$, or all equal to $\tL$.
We need to find another integer $A$ lying in a prescribed residue class, having the sharper property that all the digits in the smaller interval $[a,b)$ are \emph{equal} to $\tO$ (or any other digit combination on $[a,b)$.
At this point, Schlickewei's $p$-adic subspace theorem enters in an essential way.
Lemma~\ref{lem_not_too_many_digits} yields
\[
L\bigl(
C3^K
\bigr)
\leq\kappa+o(K),
\]
and we choose $K$ large enough (depending on $\eta,\varepsilon,k$), such that $L\bigl(C3^K\bigr)<b-a+\varepsilon K$ and $\varepsilon K\geq m$.
Assume for a moment that $(C2^K)^{[a-m,b)}=0$. 
Since $a\geq \varepsilon K$,
there exists a maximal position $c\in\{0,\ldots,a-m-1\}$ such that 
$\delta_c(C3^K)=1$. (It follows that $\delta_{c+1}(C3^K)=0$).
We ``shift'' this appearance of the digit $\tL$ by $a-m-c<\varepsilon K$ places, to the position $a-m$.
Setting $B=2^{a-m-c}C$, we obtain
\[(B3^K)^{[a-m,b)}=2^{a-m}.\]
For any given $M\ge0$, the block of digits of $rkB3^k+M$ with indices in the interval $[a,b)$ changes step by step, as $r$ is varied. More precisely, since we chose $2^{m-1}\leq k<2^m$, this block attains each value once or twice, cycling through all $2^{b-a}$ possibilities.
\emph{In particular},
we may choose $r<2^{b-a+1}$ in such a way that
\[\bigl(rkB3^K+d3^K\bigr)^{[a,b)}=\omega\]
(where we can assume without loss of generality that $0\le d<k$).
We set $A\eqdef rkB+d\in d+k\mathbb N$.
Collecting the estimates, we have $C\leq 2k2^{\eta K}$, $B\leq 2^{\varepsilon K}C$, $r+1\leq 2^{\eta K+1}$, and therefore
$A\leq 4k^22^{2\eta K+\varepsilon K}$.


In the case that $(C2^K)^{[a-m,b)}=2^{\kappa}-1$,
we choose the maximal position $c\in\{0,\ldots,a-m-1\}$ such that
$\delta_c(C3^K)=0$ (and $\delta_{c+1}(C3^K)=1$).
As $r$ runs, the digits of $rkB3^K$ in $[a,b)$ cycle through all possibilities \emph{in the opposite direction}, and we obtain the conclusion in a completely analogous way.
\end{proof}

\begin{proof}[Proof of Proposition~\ref{prp_1}]
Assume that $N\ge1$, and choose $\nu\ge0$ in such a way that
\begin{equation}\label{eqn_nu_choice}
2^{\nu}\leq N<2^{\nu+1}.
\end{equation}
In the digit elimation procedure below we will use a parameter
\begin{equation}\label{eqn_R_choice}
R=2^m,\quad\mbox{where}\quad
m=\lfloor \nu/10\rfloor.
\end{equation}
The parameter $R$ will be used in van der Corput's inequality;
its binary length is a fraction $\sim 1/10$ of the binary length of $N$.

\bigskip\noindent\textbf{Applying van der Corput's inequality.}
Let us first introduce an additional factor detecting whether $n$ is even or odd.
Using Iverson bracket notation for convenience, We have
\[\Sone^{(i)}
=
\frac1N
\sum_{0\leq n<N}
\e\Bigl(t_1\digitsum_2\bigl(n3^K\bigr)+t_2\digitsum_3(n)\Bigr)
\biglbb n\equiv i\bmod 2\bigrbb,
\]
therefore an application of van der Corput's inequality yields
\begin{equation}\label{eqn_Sone_Mzero}
\bigl\lvert S_1^{(i)}\bigr\rvert^2\ll
\frac 1{R} \sum_{\substack{1\leq r<R\\2\mid r}}\Mzero^{(i)}+\Ezero,
\end{equation}
where
\begin{equation}\label{eqn_Mzero_def}
\Mzero^{(i)}\eqdef
\frac1N
\sum_{n<N}
\e\Bigl(t_1\digitsum_2\bigl(n3^K\bigr)-t_1\digitsum_2\bigl((n+r)3^K\bigr)
+t_2\digitsum_3(n)-t_2\digitsum_3(n+r)\Bigr)
\biglbb n\equiv i\bmod 2\bigrbb
\end{equation}
and
\begin{equation}\label{eqn_Ezero_def}
\Ezero\eqdef\frac 1R+\frac RN.
\end{equation}
We introduce a new parameter $\lambda_3$ to be chosen in a moment (see~\eqref{eqn_c1Rlambda_choice} below),
and apply the ``carry lemma''~\cite{DrmotaMauduitRivat2019,MauduitRivat2009,MauduitRivat2010,MuellnerSpiegelhofer2017,Spiegelhofer2015} in order to replace
$\digitsum_3$ by a $3^{\lambda_3}$-periodic term.
Setting
\begin{equation}\label{eqn_Eone_def}
\Eone\eqdef\frac R{3^{\lambda_3}},
\end{equation}
we obtain
\begin{equation}\label{eqn_Mzero_estimate}
\Mzero^{(i)}=
\frac1N
\sum_{\substack{0\leq n<N\\n\equiv i\bmod 2}}
\e\Bigl(t_1\digitsum_2
\bigl(n3^K\bigr)-t_1\digitsum_2
\bigl((n+r)3^K\bigr)
+t_2\digitsum^{[0,\lambda_3)}_3(n)-t_2\digitsum^{[0,\lambda_3)}_3(n+r)\Bigr)
+\LandauO(\Eone).
\end{equation}
Writing $n=\nL3^{\lambda_3}+\nO$, where $0\leq \nO<3^{\lambda_3}$,
we obtain
\begin{equation}\label{eqn_Mzero_estimate_1}
\Mzero\ll
\sum_{0\leq n_0<3^{\lambda_3}}
\bigl\lvert\Mtwo(n_0)\bigr\rvert
+\LandauO(\Eone),
\end{equation}
where
\begin{align}
\Mtwo(n_0)&\eqdef
\frac1N
\sum_{\substack{0\leq n<N\\
n\equiv i\bmod 2\\
n\equiv \nO\bmod 3^{\lambda_3}} }
\e\Bigl(t_1\digitsum_2
\bigl(n3^K\bigr)-t_1\digitsum_2
\bigl(n3^K+r3^K\bigr)\Bigr).
\end{align}
Note that the $3^{\lambda_3}$-periodic terms $\digitsum^{[0,\lambda_3)}_3$ have vanished.


\bigskip
\noindent
We distinguish between the cases ``$K$ small'' and ``$K$ large''.
For this, we choose parameters $c_1,R,\lambda_3$ in such a way that 
\begin{equation}\label{eqn_c1Rlambda_choice}
c_1=\min\biggl\{c_2,\frac1{10}\frac{\log 2}{\log 3}\biggr\},\quad
R\asymp N^{1/10},\quad
3^{\lambda_3}\asymp N^{1/5}.
\end{equation}
We consider, respectively, the cases
\[\begin{cases}
K<c_1\log_2 N,&\mbox{and}\\
c_1\log_2 N\leq K<c_2\log_2N.
\end{cases}\]
Let us start with the easier case.
\subsection{Small values of $K$}
Note that our choice of $\lambda_3$ implies that
$3^K\leq N^{1/10}$.
Choose
\begin{equation}\label{eqn_lambda2_choice}
\lambda_2=\lfloor3\nu/10\rfloor,
\end{equation}
such that
\begin{equation}\label{eqn_R_small}
\frac{R3^K}{2^{\lambda_2}}\ll N^{1/10}.
\end{equation}
Similarly to~\eqref{eqn_Mzero_estimate},
we have
\begin{align*}
\Mtwo(n_0)=
\frac1{2^{\lambda_2}3^{\lambda_3}}
\sum_{0\leq n<2^{\lambda_2}}
\e\Bigl(t_1
\digitsum^{[0,\lambda_2)}
(n)-t_1\digitsum^{[0,\lambda_2)}_2
\bigl(n+r3^K\bigr)\Bigr)
+\Efour,
\end{align*}
where
\begin{equation}\label{eqn_Efour_def}
\Efour\eqdef{R3^K}\biggl(\frac1{2^{\lambda_2}}\frac1{3^{\lambda_3}}+\frac1N\biggr).
\end{equation}
For this, we just note that the lowest $\lambda_2$ binary digits of $n3^K$ attain each value in a periodic fashion, as $n$ runs through $i3^{\lambda_3}+2\cdot 3^{\lambda_3}\mathbb Z$.
Excluding $R3^K$ digit combinations, the estimate follows.
Next, we replace $\digitsum^{[0,\lambda_2)}$ by $\digitsum$ again, reusing the fact that only a proportion $\ll R3^K/2^{\lambda_2}$ of integers $n$ satisfy
\[\digitsum(n)-\digitsum\bigl(n+r3^K\bigr)\neq
\digitsum^{[0,\lambda_2)}(n)-\digitsum^{[0,\lambda_2)}\bigl(n+r3^K\bigr).
\]
As this error is swallowed by $\Efour$, we obtain
\begin{equation}\label{eqn_Szero_estimate}
\bigl\lvert\Sone^{(i)}\bigr\rvert^2\ll
\frac1R\sum_{1\leq r\leq R}\bigl\lvert
\gamma_{r3^K}(t_1)
\bigr\rvert
+\Ezero+\Eone+\Efour,
\end{equation}
where
\begin{equation}\label{eqn_gamma_def}
\gamma_t(\vartheta)\eqdef
\lim_{M\rightarrow\infty}
\frac1M
\sum_{0\leq n<M}
\e\Bigl(t_1
\digitsum
(n)-t_1\digitsum
\bigl(n+t\bigr)\Bigr).
\end{equation}

We recall a result from the paper
\cite{Spiegelhofer2022}
by the second author, which was also used in~\cite{SpiegelhoferWallner2023}.
\begin{lemma}[{\cite[Lemma~2.7]{Spiegelhofer2022}}]\label{lem_gamma_tail}
Assume that $t\geq 1$ has at least $M'=2M+1$ blocks of $\tL$s.
Then
\begin{equation*}
\left\lvert\gamma_t(\vartheta)\right\rvert\leq
\biggl(1-\frac{\lVert\vartheta\rVert^2}2\biggr)^{M}
\leq
\exp\biggl(-\frac{M\lVert\vartheta\rVert^2}2\biggr).
\end{equation*}
\end{lemma}
In order to guarantee that $r3^K$ has sufficiently many blocks of $\tL$s, we apply Lemma~\ref{lem_not_too_many_digits}.
We assume that
\begin{equation}\label{eqn_R_size}
R<2^{\eta K}.
\end{equation}
For each $\varepsilon>0$, and $K\ge K_0(\eta,\varepsilon)$, the lemma implies
\[
L\bigl(r3^K\bigr)\leq(1+\varepsilon)\eta K.\]
In particular, as $K\rightarrow\infty$, and~\eqref{eqn_R_size} is satisfied, the number of maximal blocks of $\tL$s in $r3^K$ is $\gg K\eta$ .

By Lemma~\ref{lem_gamma_tail} it follows that there exists a constant $c$ (depending on $\eta$) such that
\[\bigl\lvert \Sone^{(i)}\bigr\rvert^2
\ll\exp\biggl(
-cK\lVert t_1\rVert^2
\biggr)
+\Ezero+\Eone+\Efour.
\]
Collecting the error terms, we see tht the proposition is proved for the case of ``small $K$''.

\subsection{Large values of $K$}
In order to handle the second case, we will use the odd elimination lemma (Corollary~\ref{cor_odd}), based on Schlickewei's $p$-adic subspace theorem in an essential way.

\bigskip\noindent\textbf{Iterating van der Corput.}
Applying Cauchy-Schwarz and van der Corput alternatingly, we arrive at the following statement~\cite{Spiegelhofer2023cubes}.

\begin{lemma}\label{lem_vdC_iterated}
Let $Q\geq 1$ be an integer.
Assume that $J$ is a finite nonempty interval in $\mathbb Z$,
and $g:J\rightarrow \{z\in\mathbb C:\lvert z\rvert=1\}$.
For all integers $\factor_0,\ldots,\factor_{Q-1}\ge1$ and $R\ge1$, we have
\begin{equation}\label{eqn_vdC_iterated}
\begin{aligned}
\Biggl\lvert
\frac1{\lvert J\rvert}
\sum_{n\in J}
g(n)
\Biggr\rvert^{2^Q}
&\ll
\frac1{R^Q}
\sum_{r\in\{1,\ldots,R-1\}^Q}
\bigl\lvert
K\bigl(r_0\factor_0,\ldots,r_{Q-1}\factor_{Q-1}\bigr)
\bigr\rvert
\\&\quad
+
\frac{\bigl(\factor_0+\cdots+\factor_{Q-1}\bigr)R}{\lvert J\rvert}
+\frac1R,
\end{aligned}
\end{equation}
where
\[
K\bigl(t_0,\ldots,t_{Q-1}\bigr)
\eqdef
\frac1{\lvert J\rvert}
\sum_{n\in J}
\prod_{\varepsilon\in\{0,1\}^Q} 
\mathcal C^{\lvert\varepsilon\rvert}
g\Biggl(n+\sum_{0\leq\ell<Q}\varepsilon_\ell t_\ell\Biggr),
\]
and $\mathcal C$ is pointwise complex conjugation.
The implied constant depends only on $Q$.
\end{lemma}
Each factor $\factor_\ell$ is responsible for the ``elimination of digits'' in a short interval, 
and also in a small margin just below the interval.
Recall that $m$ is the binary length of $R$~\ref{eqn_R_choice}.
We double this value in order to obtain a useful \emph{double margin}, and we define
\begin{equation}\label{eqn_elimination_intervals_def}
\begin{array}{rlrlrl}
a_\ell&=\lambda_2-\ell\kappa&
b_\ell&=a_\ell-m&
c_\ell&=a_\ell-2m,\\
I_\ell&=[a_{\ell},a_{\ell-1}),&
I'_\ell&=[b_{\ell},a_{\ell-1}),&
I''_\ell&=[c_{\ell},a_{\ell-1}).
\end{array}
\end{equation}
The interval of digits we want to remove in step $\ell$ (where $1\leq \ell\leq Q$, and $Q$ is yet to be defined) is $I_\ell$.
In order to achieve this, we \emph{clear} the binary digits of $3^{\Kone}$ in the larger interval $I''_\ell$ of length $\kappa+2m$, multiplying this value by some odd integer produced by Corollary~\ref{cor_odd}.
Multiplying this new value by any number in $r_\ell\in\{0,\ldots,R-1\}$, we still have only $\tO$s 
in the slightly smaller interval $I'_\ell$ of length $\kappa+m$.
Consequently, \emph{in most cases}, the quantity $r_\ell \factor_\ell3^{\Kone}$ can be added to an integer $n$ without changing the digits in $I_\ell$.
Namely, this happens precisely when not all of the digits of $n$ on the upper part $[b_{\ell},a_{\ell})$ of the margin are equal to $\tL$, and thus carry propagation is interrupted.
Set
\[g(n)=\e\bigl(t_1\digitsum_2^{[0,\lambda_2)}\bigl(n3^{\Kone}+\nO3^K\bigr)\bigr).\]
Lemma~\ref{lem_vdC_iterated} implies
\begin{equation}\label{eqn_Mtwo_Mthree}
\begin{aligned}
\frac1{\lvert J(n_0)\rvert}
\bigl\lvert\Mtwo(n_0)\bigr\rvert^{2^{Q}}
&\leq 
\frac1{R^{Q-1}}
\sum_{r_1,\ldots,r_{Q}\in\{1,\ldots,R-1\}}
\bigl\lvert
\Mthree\bigl(r_03^K,r_1\factor_13^{\Kone},\ldots,r_{Q}\factor_{Q}3^{\Kone};\nO3^K\bigr)
\bigr\rvert
\\&\quad
+
\frac{\bigl(\factor_1+\cdots+\factor_{Q}\bigr)R}{\lvert J\rvert}
+\frac1R,
\end{aligned}
\end{equation}
where 
$\Mthree$ is the iterated correlation
\begin{equation}\label{eqn_before_CLT}
\begin{aligned}
\Mthree(t_0,\ldots,t_Q;a)\eqdef
\frac1{\lvert J(\nO)\rvert}
\sum_{\nL\in J(\nO)}
\prod_{\varepsilon\in\{0,1\}^{Q+1}} 
\e\Biggl(
(-1)^{\lvert\varepsilon\rvert}
t_1s_2^{[0,\lambda_2)}\biggl(\nL 3^{\Kone}+\sum_{0\leq\ell\leq Q}\varepsilon_\ell t_\ell+a\Biggr)\Biggr).
\end{aligned}
\end{equation}
In the next few paragraphs, we will choose the parameters $\factor_\ell$ suitably in order to discard intervals of digits, one interval for each $\ell\in\{1,\ldots,Q\}$.

\bigskip\noindent\textbf{Preparing digit elimination.}
We begin with the observation that the binary digits of $\nL\mapsto \nL3^{\Kone}$ with indices in $[b_\ell,a_\ell)$ attain all combinations in a uniform manner.
This will be needed in order to bound the number of cases where for some $\ell$, carry propagation from the upper half $[b_\ell,a_\ell)$ of the margin into the interval $I_\ell$ happens.
This is our first use of Corollary~\ref{cor_odd}.
Assume that $\varepsilon,\eta>0$.
We suppose in the following that $x,y,\Kone$ are integers satisfying
\begin{equation}\label{eqn_elim_requirement}
\begin{aligned}
\Kone&\geq\Kzero(\eta,\varepsilon,1),\\
\varepsilon \Kone&\leq x\leq y\leq x+\eta \Kone.
\end{aligned}
\end{equation}
By the corollary, 
\begin{equation}\label{eqn_elim_statement}
\bigl(A3^{\Kone}\bigr)^{[x,y)}=1
\quad\mbox{for some}\quad
A\ll 2^{(2\eta+\varepsilon)\Kone}.
\end{equation}
That is, the digits of $A3^{\Kone}$ are zero on $(x,y)$,
and the digit at index $x$ equals $\tL$.
It follows that for each $b$,
\[\nLL\mapsto\bigl((\nLL A+\nLO)3^{\Kone}+b\bigr)^{(x,y)}\]
attains each value once or twice in a row, running through all possibilities.
Using the decomposition $\nL=\nLL A+\nLO$, we obtain,
for any $\omega\in\{0,\ldots,2^{y-x}-1\}$,
\[
\bigl\lvert\bigl\{
\nL\in J(\nO):
\bigl(\nL3^{\Kone}+b\bigr)^{[x,y)}=\omega
\bigr\}
\bigr\rvert
\ll\Etwo
\]
with some absolute implicit constant, where
\begin{equation}\label{eqn_Etwo_def}
\Etwo\eqdef\frac{\lvert J(\nO)\rvert}{2^{y-x}}+A.
\end{equation}

Recall that in our application, the role of interval $[x,y)$ will be played by the upper part $[b_{\ell},a_{\ell})$ of the margin.
\emph{For all} $(\varepsilon_0,\ldots,\varepsilon_Q)\in\{0,1\}^{Q+1}$,
the digits of
\begin{equation}\label{eqn_2Q_possibilities}
\mathcal L(\varepsilon_0,\ldots,\varepsilon_Q)\eqdef
\nL3^{\Kone}+\varepsilon_0\rzero+\sum_{1\leq\ell\leq Q}\varepsilon_\ell r_\ell\factor_\ell3^{\Kone}+\nO3^{K}
\end{equation}
with indices in $[b_{\ell},a_{\ell})$
should not be identical to $\tL$.
That is, $\omega=2^{y-x}-1$ has to be excluded.
The error term $\Etwo$ has to be multiplied by a factor $Q2^{Q+1}$, since we want to avoid carry overflow \emph{simultaneously} for all
$1\leq \ell\leq Q$ and $\varepsilon\in\{0,1\}^{Q+1}$.

\smallskip\noindent
Consequently, for all $\ell\in\{1,\ldots,Q\}$ and all choices $\varepsilon_k\in\{0,1\}$ for $k\in\{0,\ldots,Q\}\setminus\{\ell\}$, the integers
\begin{equation}\label{eqn_whatwehave}
\mathcal L\bigl(\varepsilon_0,\ldots,\varepsilon_{\ell-1},0,\varepsilon_{\ell+1},\ldots,\varepsilon_Q\bigr)\quad\mbox{and}\quad
\mathcal L\bigl(\varepsilon_0,\ldots,\varepsilon_{\ell-1},1,\varepsilon_{\ell+1},\ldots,\varepsilon_Q\bigr)
\end{equation}
have the same digits in 
$[a_{\ell},a_{\ell-1})$
for all but $\LandauO(\Ethree)$
integers $\nL\in J(\nO)$, where
\begin{equation}\label{eqn_Ethree_def}
\Ethree\eqdef 
\frac{N}{3^{\lambda_3}2^{y-x}}+A.
\end{equation}
The implied constant may depend on $Q$.

\bigskip\noindent\textbf{Discarding digits block by block.}
Assume that $c_2>0$, and 
consider the requirement
$c_1\log_2N\leq K\leq c_2\log_2 N$. 
In step $\ell$, where $1\le\ell\le Q$, we want to remove the interval 
$[a_\ell,a_{\ell-1})$ (see~\eqref{eqn_elimination_intervals_def}).
For simplicity, we assume that its length $\kappa$ is a fixed fraction of the binary length $\nu$ of $N$:
\begin{equation}\label{eqn_elimination_interval_length}
\kappa\sim \nu/10.
\end{equation}
After the removal of $I_1,\ldots,I_Q$, only the interval $[0,\mu)$ of digits will be left, where
\begin{equation}\label{eqn_digits_left}
\mu=\lambda_2-Q\kappa
\end{equation}
should satisfy
\[2\kappa\leq \mu<3\kappa.\]
Clearly, this implies
\begin{equation}\label{eqn_Q_size}
Q\asymp (1+c_2)^{-1}.
\end{equation}
As for the size of the margins, we note that $m=\kappa=\log_2R$.
By~\eqref{eqn_digits_left}, the intervals
\[I''_{\ell}=\bigl[a_\ell-2m,a_{\ell-1}\bigr),\]
where $a_\ell=\lambda_2-\ell\kappa$ as in~\eqref{eqn_elimination_intervals_def},
are well separated from $0$.
In particular, there are enough digits below our intervals to be eliminated in order to apply Corollary~\ref{cor_odd}.
Suppose that $\Kone\geq \Kzero(\eta,\varepsilon,2)$ as stated in the corollary.
(The parameter $2$ handles the even/odd restriction on $n$.)
We obtain \emph{odd factors} $\factor_1,\ldots,\factor_Q$,
\[\factor_{\ell}\ll 2^{(2\eta+\varepsilon)\Kone},
\]
such that $\bigl(\factor_{\ell}3^{\Kone}\bigr)^{[x,y)}=0$,
where $[x,y)=I''_{\ell}$.


Introducing the error $\Ethree$ defined in~\eqref{eqn_Ethree_def} takes care of the integers $\nL\in J(\nO)$ that are exceptional for \emph{some} index $1\leq \ell\leq Q$ and \emph{some} choice $(\varepsilon_0,\ldots,\varepsilon_Q)\in\{0,1\}^{Q+1}$.
For the remaining $\nL$, we consider the product on the right hand side of~\eqref{eqn_before_CLT}.
In a way analogous to~\cite{DrmotaMuellnerSpiegelhofer2022, Spiegelhofer2020} we
may discard the digits with indices outside $[0,\mu)$.
That is, we (1) exclude the critical indices $\nL$, (2) apply the digit-cancelling argument implemented in the cited papers,
and (3) reinsert the missing indices again.
Up to an error term $\Ethree$, this leads to an expression
\begin{equation}\label{eqn_Mfour}
\begin{aligned}
&\Mfour(t_0,\ldots,t_Q;a)\eqdef
\frac1{\lvert J(\nO)\rvert}
\sum_{\nL\in J(\nO)}
\prod_{\varepsilon\in\{0,1\}^{Q+1}} 
\e\Biggl(
(-1)^{\lvert\varepsilon\rvert}
t_1s_2^{[0,\mu)}\biggl(\nL 3^{\Kone}+\sum_{0\leq\ell\leq Q}\varepsilon_\ell t_\ell+a\Biggr)\Biggr),\\
&\mbox{where\quad $t_0=\rzero3^K$, \quad $t_\ell=r_\ell\factor_\ell3^{\Kone}$ for $1\leq \ell\leq Q$,\quad and \quad $a=\nO3^K$.}
\end{aligned}
\end{equation}
At this point, the main work has already been done.
The remaining sum over $\nL$ is long enough to traverse all digit combinations on $[0,\mu)$ in a uniform manner, but by construction, the sums over $r_\ell$ are too short.
In order to transform the higher order correlations into a Gowers norm, we shorten our interval of digits once more.
This leaves only $[0,\rho)$, where $\rho=\nu/6$ (recall that $\mu\ge2\kappa=\nu/5$).

\bigskip\noindent\textbf{Removing the last interval of digits.}
Let the odd positive integer $\factor_{Q+1}<2^\mu$ be chosen in such a way that
\begin{equation}\label{eqn_multiplicative_inverse}
\factor_{Q+1}3^{\Kone}\equiv 1 \bmod 2^\mu.
\end{equation}
Let $\rho'$ be an integers such that $0\leq\rho'\leq\rho$.
We apply van der Corput's inequality one last time in order to eliminate the digits in $[\rho,\mu)$.
This yields
\begin{equation}\label{eqn_Mfour_Mfive}
\begin{aligned}
\hspace{8em}&\hspace{-8em}
\bigl\lvert
\Mfour(t_0,\ldots,t_Q;a)
\bigr\rvert^2
\leq
\frac{\lvert J(\nO)\rvert+2^{\mu+\rho'}}{\lvert J(\nO)\rvert2^{\rho'}}
\sum_{0\leq \lvert r_{Q+1}\rvert<2^{\rho'}}
\biggl(1-\frac{\lvert r_{Q+1}\rvert}{2^{\rho'}}\biggr)
\\&\times\sum_{\nL,\nL+r_{Q+1}\factor_{Q+1}\in J(\nO)}
\Mfive(t_0,\ldots,t_{Q+1};\sigma,3^{\Kone},a),
\end{aligned}
\end{equation}
where
\begin{equation}\label{eqn_Mfive}
\begin{aligned}
&\Mfive(t_0,\ldots,t_{Q+1};\sigma,x,a)\eqdef
\frac1{\lvert J(\nO)\rvert}
\sum_{\nL\in J(\nO)}
\prod_{\varepsilon\in\{0,1\}^{Q+2}} 
\e\Biggl(
(-1)^{\lvert\varepsilon\rvert}
{t_1}s_2^{[0,\sigma)}\biggl(\nL x+\sum_{0\leq\ell\leq Q+1}\varepsilon_\ell t_\ell+a\Biggr)\Biggr),\\
&\mbox{and\quad
$\sigma=\mu$,\quad
$x=3^{\Kone}$,
\quad $t_0=\rzero3^K$, \quad $t_\ell=r_\ell\factor_\ell3^{\Kone}$ for $1\leq \ell\leq Q+1$, and \quad $a=\nO3^K$.}
\end{aligned}
\end{equation}
Note that~\eqref{eqn_multiplicative_inverse} implies
\[\bigl(r_{Q+1}\factor_{Q+1}3^{\Kone}\bigr)^{[\rho',\mu)}=0.\]
Consequently, for any shift $s$ and any interval $I$ of length $2^\rho$, there are at most $2^{\rho'}$ integers $\nL\in I$ such that
\[
\bigl(\nL3^{\Kone}+s\bigr)^{[\rho,\mu)}\neq
\bigl((\nL+r_{Q+1}\factor_{Q+1})3^{\Kone}+s\bigr)^{[\rho,\mu)}
\]
(the critical values being those where all digits of $\nL3^{\Kone}+s$  in $[\rho',\rho)$ are equal to $\tL$).
Therefore, excluding $\LandauO(\Ethree')$ integers $\nL\in J(\nO)$, where
\[\Ethree'=2^{\rho'}\biggl\lceil\frac{\lvert J(\nO)\rvert}{2^{\rho}}\biggr\rceil
\leq 2^{\rho'-\rho}\lvert J(\nO)\rvert+2^{\rho'},
\]
we may replace $[0,\mu)$ by $[0,\rho)$ in the definition of $\Mfive$.
Note once again that the implied constant may depend on $Q$.

The next step consists in replacing the sum over multiples of $3^{\Kone}$ by a full sum, exploiting the fact that the sum over $\nL$ is long enough, and $3^{\Kone}$ is odd.
Also, we may remove the factors $t_\ell$, for $0\leq \ell\leq Q$,
using $R=2^m\mid 2^{\rho}$.
Introducing a negligible error term coming from the decomposition of $J(\nO)$ into intervals of length $2^\rho$ (the last interval might contribute an error), we arrive at the expression
\begin{equation}\label{eqn_Msix}
\begin{aligned}
&\Msix(t_0,\ldots,t_Q;a)\eqdef
\frac1{2^{\rho'}}
\sum_{r_{Q+1}<2^{\rho'}}
\frac1{2^{(Q+1)\rho}}
\sum_{r_0,\ldots,r_{Q+1}<2^{\rho}}
\Biggl\lvert
\Mfive(t_0,\ldots,t_Q,t_{Q+1};\rho,1,0)
\Biggr\rvert
,\\
&\mbox{where\quad $t_0=\cdots=t_Q=1$, \quad $t_{Q+1}=\factor_{Q+1}3^{\Kone}$.}
\end{aligned}
\end{equation}
We have almost arrived at a Gowers norm.
In order to handle the summation over $r_{Q+1}$, which is shorter than $2^\rho$ by construction, we use a standard trick reminiscent of the ``$17$ camels puzzle''~\cite{Stockmeyer2013}.
We extend the sum to $\{0,\ldots,2^\rho-1\}$ (thus inserting terms representing the $18$th camel), noting that each summand is nonnegative.
Applying the the Cauchy--Schwarz inequality, we insert another variable $r_{Q+2}$ ranging over $\{0,\ldots,2^\rho-1\}$, and discard the absolute value again (this is also applied in our papers~\cite{Spiegelhofer2020,DrmotaMuellnerSpiegelhofer2022}).
The expression $\Msix$ is replaced by a Gowers-$(Q+3)$-norm on $\mathbb Z/2^\rho\mathbb Z$, where an additional factor $\bigl(2^{\rho-\rho'}\bigr)^2$ arising from our extension of a summation range is present.
Having arrived at a full Gowers norm, we can use available techniques in order to arrive at a nontrivial estimate.
More precisely, extending the method devised in~\cite[Chapter~5.2]{DrmotaMuellnerSpiegelhofer2022} for the \emph{Zeckendorf sum-of-digits function}, we arrive at the following statement.
\begin{proposition}
Let $q\ge2$ and $Q\ge1$ be integers, and $\vartheta\in\mathbb Z$.
There exists a constant $c>0$ such that
\begin{equation}
\frac1{q^{(Q+1)\mu}}
\sum_{0\leq n<q^\mu}
\sum_{r\in \{0,\ldots,q^\mu-1\}^Q}
\prod_{\varepsilon\in\{0,1\}^Q}
\e\biggl(
(-1)^{\lvert\varepsilon\rvert}
\vartheta s_q\bigl(n+\varepsilon\cdot r\bigr)
\biggr)
\ll q^{-c\mu\lVert (q-1)\vartheta\rVert^2}
\end{equation}
as $\mu\rightarrow\infty$.
\end{proposition}

A full proof of this statement is given in a manuscript by Jelinek (in preparation).
Provided that $\rho'$ is close enough to $\rho$,
the gain coming from this estimate is strictly bigger than the contribution $(2^{\rho-\rho'})^2$ of the artificially added terms.
Thus, ``the additional camel can be returned'', leaving us with a nontrivial estimate of the original sum.

We have to take care of the fact that $\lVert\vartheta\rVert$ may be very small, yielding only a small gain in the Gowers norm.
Correspondingly, our additional margin $\rho-\rho'$ has to be very small too.
This issue has been dealt with in our paper~\cite{DrmotaMuellnerSpiegelhofer2022} with M\"ullner, see pages 80--81, and only amounts to decreasing the constant $c$ in the expression $N^{-c\lVert \vartheta\rVert^2}$ by some factor.


\end{proof}

\section{Proof of Proposition~\ref{Pro2}}

The proof of Proposition~\ref{Pro2} relies mainly on the following lemma.

\begin{lemma}\label{LeKey}
Suppose that $0 \le K \le c_2 \log N$, where $c_2 >0$.
Set $N_1 = \lfloor \log_{2 } (3^K N) \rfloor$ and $N_2 =  \lfloor \log_{3 }N \rfloor$ 
let $\lambda>0,\eta>0$ be an arbitrary constant and let $d_1$, $d_2$ be positive integers. 
Then for integers
\begin{equation}\label{eqrange2}
N_1^{\eta} \le k_1 < k_2 < \cdots < k_{d_1} \le  N_1 - N_1^{\eta}
\end{equation}
and
\begin{equation}\label{eqrange2-2}
N_2^{\eta} \le \ell_1 < \ell_2 < \cdots < \ell_{d_2} \le  N_2 - N_2^{\eta}
\end{equation}
we have for $i\in \{0,1\}$, as $N\to\infty$
\begin{align}
&\frac 2N \# \left\{ n < N,\, n\equiv i\bmod 2 \,:\, \varepsilon_{2,k_{j_1}}(3^Kn) = b_{j_1},\,
1\le j_1\le d_1,\, \varepsilon_{3,\ell_{j_2}}(n) = c_{j_2},\, 1\le j_2 \le d_2  \right\}
\nonumber\\
&= \frac 1{2^{d_1}3^{d_2}} + O\left((\log N)^{-\lambda}\right)
\label{eqPro21}
\end{align}
and (for $r\in \{0,1\}$)
\begin{align}
&\frac 2N \# \left\{ n < N \, n\equiv i\bmod 2 \,:\, \varepsilon_{2,k_{j_1}}(n) = b_{j_1},\,
1\le j_1\le d_1,\, \varepsilon_{3,\ell_{j_2}}(2^Kn+r) = c_{j_2},\, 1\le j_2 \le d_2  \right\}
\nonumber\\
&=
 \frac 1{2^{d_1}3^{d_2}} + O\left((\log N)^{-\lambda}\right)
\label{eqPro21-2}
\end{align}
uniformly for $0\le K \le c_2\log N$, $b_{j_1} \in \{0,1\}$, $c_{j_2}\in \{0,1,2\}$  and $k_{j_1}$, $\ell_i$ in the given ranges,
where the implicit constant of the error term may depend on 
$h_1$, $h_2$, and on $\lambda$.
\end{lemma}
We will prove this lemma in Section~\ref{sec:LeKey} with the help of exponential sum estimates
provided in Section~\ref{sec:exponentialsums}.

The proof of Proposition~\ref{Pro2} is then given in Section~\ref{sec:proofPro2}.

\subsection{Exponential Sums}\label{sec:exponentialsums}

The essential part of the proof of Lemma~\ref{LeKey} are upper bounds 
for exponential sums of the form
\begin{equation}\label{eqSexpsumdef}
S = \sum_{n< N,\ n\equiv i \bmod 2} e \left( \left( 3^K \sum_{j_1=1}^{d_1} h_{j_1} 2^{-k_{j_1}-1} +   \sum_{j_2=1}^{d_2} r_{j_2} 3^{-\ell_{j_2}-1} \right) n \right),
\end{equation}
where $h_{j_1}$ are integers not divisible by $2$, and $r_{j_2}$ 
are integers not divisible by $3$ and
$h_{j_1}$ and $r_{j_2}$ 
are absolutely upper bounded by $(\log N)^{\lambda_0}$ for some $\lambda_0> 0$.
For the sake of shortness we only prove the relation (\ref{eqPro21}). 
The corresponding relation (\ref{eqPro21-2}) can be proved in the same way
by interchanging the r\^oles of $2$ and $3$.

Clearly we have
\begin{equation}\label{eqexpsumtrivial}
\left| \sum_{n< N,\, n \equiv i \bmod 2} e(\alpha n)\right| \le \frac 1{2\|2\alpha\|}.
\end{equation}
Hence we have to find lower bounds for $\| 2\alpha \|$, where
\[
\alpha = 3^K \sum_{j_1=1}^{d_1} h_{j_1} 2^{-k_{j_1}-1} +   \sum_{j_2=1}^{d_2} r_{j_2} 3^{-m_{j_2}-1}.
\]
Actually, we will prove that (uniformly under the above mentioned assumptions) 
\begin{equation}\label{eqalphalowerbound}
\| 2 \alpha \| \gg    \frac{ (\log N)^\lambda} N
\end{equation}
for any given constant $\lambda> 0$. This implies then the following property.

\begin{lemma}\label{Leexponentialsums}
Suppose that $0 \le K \le c_2 \log N$, where $c_2 >0$.
Set $N_1 = \lfloor \log_{2 } (3^K N) \rfloor$ and $N_2 =  \lfloor \log_{3 }N \rfloor$ 
let $\lambda_0>0,\lambda>0$ be arbitrary constants and let $d_1$, $d_2$ be positive integers. 
Then for integers
\[
N_1^{\eta} \le k_1 < k_2 < \cdots < k_{d_1} \le  N_1 - N_1^{\eta},
\]
\[
N_2^{\eta} \le \ell_1 < \ell_2 < \cdots  < \ell_{d_2} \le  N_2 - N_2^{\eta}
\]
and for odd integers $h_1,\ldots,h_{d_2}$ and integers $r_1,\ldots, r_{d_2}$ that are
not divisible by $3$ with 
\[
\max_{1\le j_1\le d_1} |k_{j_1}| \le (\log N)^{2\lambda_0} \quad \mbox{and}\quad
\max_{1\le j_2\le d_2} |r_{j_2}| \le (\log N)^{2\lambda_0}
\]
we have the uniform upper bouund
\begin{equation}\label{eqSbound}
\max |S| \ll N (\log N)^{-\lambda},
\end{equation}
where $S$ denotes the exponential sum (\ref{eqSexpsumdef}).
\end{lemma}

We will distinguish between several cases. 

\subsubsection{$d_1 = 0$}

In this case $\alpha$ simplifies to 
\[
\alpha = \sum_{j_2=1}^{d_2} r_{j_2} 3^{-\ell_{j_2}-1}.
\]
By assumption (\ref{eqrange2-2}) we have $N_2^\eta \le \ell_{j_2} \le N_2 - N_2^\eta$, 
$|r_{j_2}| \le (\log N)^\lambda$ and $r_{j_2}$ is not divisible by $3$, $1\le j_2 \le d_2$. Consequently
\[
\|2\alpha\| = 2|\alpha| =  \left|2 \sum_{j_2=1}^{d_2} r_i 3^{-\ell_i-1} \right|  
\ge \frac 2{3^{\ell_{d_2} + 1}}
 \gg \frac{ e^{c (\log N)^\eta} } N \gg \frac{ (\log N)^\lambda} N
\]
for some constant $c> 0$. 

\subsubsection{$d_2 = 0$}\label{subsectiond2=0}

Here we have
\[
\alpha =  3^K \sum_{j_1=1}^{d_1} h_{j_1} 2^{-k_{j_1}-1} ,
\]
where $N_1^\eta \le k_{j_1} \le N_1 - N_1^\eta$, 
$|h_{j_1}| \le (\log N)^\lambda$ and $h_{j_1}$ is not divisible by $2$. Recall that $N_1 = \lfloor \log_2 (3^K N) \rfloor$.
Let $H$ denote the nearest integer to $\alpha$, that is,
\[
\| 2\alpha\| = | 2\alpha - H|.
\]

In a first step we assume that $H = 0$. Here we can argue similiarly as in the previous case:
\[
\|2\alpha\| = 2|\alpha| =  3^K \left|\sum_{j_1=1}^{d_1} h_{j_1} 2^{-k_{j_1}} \right|  \ge    \frac {3^K}{2^{k_{d_1}}}
 \gg \frac{ 3^K  e^{c (\log N)^\eta} } {3^K N}  \gg \frac{ (\log N)^\lambda} N
\]
and we are done.

Next suppose that $H\ne 0$. Since $\| 2\alpha\| \le 1$ it follows that 
\[
|H| \le 1+ 3^K (\log N)^\lambda 2^{-k_1}.
\]
Set $H = 3^L H'$ with $(3,H') = 1$. Since $|H|\le 3^K$ it follows that $L\le K$. 
Furthermore set
\begin{align*}
D &= {\rm gcd} \left( 3^{K-L} h_1 2^{k_{d_1}-k_1}, 3^{K-L}  h_2 2^{k_{d_1-1}-k_1},\ldots,  3^{K-L} h_{d_1}, H' 2^{k_{d_1}+1} \right)  \\
&=  {\rm gcd} \left( h_1, h_2,\ldots, h_{d_1}, H' \right).
\end{align*}
Note that $H'$ is not divisible by $3$ and $h_d$ not by $2$. Thus, the last equality holds.
We also set
\[
h_{j_1}' = h_{j_1}/D \quad \mbox{and} \quad  H'' = H'/D.
\]
Then we have
\[
\|2\alpha\| = \frac{D 3^L}{2^{k_{d_1}}} \left| 3^{K-L} h_1' 2^{k_{d_1}-k_1} + 3^{K-L}  h_2' 2^{k_{d_1-1}-k_1} + \cdots + 
3^{K-L} h_{d_1}' -  H'' 2^{k_{d_1}+1}   \right|.
\] 
At this level we can apply Lemma~\ref{LeDio3} and we obtain 
\begin{align*}
\|\alpha\| &\gg  \frac{D 3^L}{2^{k_{d_1}+1}} \frac{ \left( 3^{K-L}  2^{k_{d_1}-k_1} \right)^{1-\delta} }
{ |h_1' \cdots h_{d_1}'\, H'' |} \\
&\gg  \frac{D 3^L}{2^{k_{d_1}+1}}  \frac{ \left( 3^{K-L}  2^{k_{d_1}-k_1} \right)^{1-\delta} }
{ (\log N)^{d_1\lambda}  3^{K-L} /(D 2^{k_1} )  } \\
&\gg \frac { 3^L} { \left( 3^{K-L}  2^{k_{d_1}-k_1} \right)^{\delta} } \\
&\gg \frac {1} { \left( 3^{K}  2^{k_{d_1}-k_1} \right)^{\delta} }.
\end{align*}
Since 
\[
3^{K}  2^{k_{d_1}-k_1} \le 3^{2K} N \le N^{1+2c_2} 
\]
we, thus, obtain
\[
\|2\alpha\| \gg N^{-\delta(1+2c_2)}.
\]
Hence by choosing $\delta = (1+2c_2)/2$ we get a proper lower bound 
$N^{-1/2} \gg (\log N)^\lambda / N$.

\subsubsection{$d_1 > 0$ and $d_2 > 0$}

As in the previous case let $H$  be the nearest integer to $2\alpha$, that is, 
$\| 2\alpha \| = |2\alpha - H|$.

First we consider the case $H=0$. Here we have
\[
\| \alpha\| =  \left| 3^K \sum_{j_1=1}^{d_1} h_{j_1} 2^{-k_{j_1}} +  
 2\sum_{j_2=1}^{d_2} r_i 3^{-m_i-1} \right|
\]
In this case we proceed precisely as in the paper \cite{Drmota2001}, where we apply Lemma~\ref{CorX} appropriately.

Without loss of generality we can assume that $h_{j_1}\ne 0$ and $r_{j_2}\ne 0$ (for all $j$ and $i$). 
Otherwise we reduce $d_1$ and $d_2$ accordingly.

We set $\delta = \eta/(d_1+d_2)$. Clearly there exists $0\le k \le d_1 +d_2-1$ such that
\[
h_{j_1+1} - h_{j_1} \not \in \left[ (\log N)^{k\delta}, (\log N)^{(k+1) \delta} \right) \quad (1\le j_1 < d_1)
\]
and
\[
r_{j_2+1} - r_{j_2} \not \in \left[ (\log N)^{k\delta}, (\log N)^{(k+1) \delta} \right)\quad (1\le j_2 < d_2).
\]

We first suppose that 
\[
h_{j_1+1} - h_{j_1} \le (\log N)^{k\delta} \quad (1\le j_1 < d_1) \quad\mbox{and}\quad
r_{i+1} - r_{j_2} \le (\log N)^{k\delta} \quad (1\le j_2 < d_2).
\]
Then we can represent $\alpha$ as
\[
\alpha = a 3^K 2^{-k_{d_2}-1} + b 3^{-\ell_{d_2}-1},
\]
where 
\[
a = \sum_{j_1=1}^{d_1} h_{j_1} 2^{k_{d_1}-k_j} \quad\mbox{and}\quad  b =  \sum_{j_2=1}^{d_2} r_{j_2} 2^{\ell_{d_2}-\ell_{j_2}} 
\]
satisfy
\[
\log|a| \ll (\log N)^{k\delta} \quad\mbox{and}\quad  \log|b| \ll (\log N)^{k\delta}.
\]
By a direct application of Lemma~\ref{CorX} we obtain
\[
2|\alpha| = \left| a 3^K 2^{-k_{d_1}} + 2b 3^{-\ell_{d_2}-1} \right| 
\ge \max \left( \left| a 3^K 2^{-k_{d_1}} \right|, \left|2 b 3^{-\ell_{d_2}-1} \right| \right)
e^{-C' \log\log N (\log N)^{k\delta} }
\]
for some constant $C' > 0$. Clearly this implies
\[
|2\alpha| \gg \frac{e^{c(\log N)^\eta}}{N} e^{-C' \log\log N (\log N)^{k\delta} }
\gg \frac{e^{c'(\log N)^\eta}}{N} \gg \frac{ (\log N)^\lambda}{N}
\]
for some constants $c>0, c'>0$.

In general we assume that for some $s_1\le d_1$ and $s_2\le d_2$
\[
h_{j_1+1} - h_j \le (\log N)^{k\delta} \quad (1\le j_1 < s_1) \quad\mbox{and}\quad
r_{j_2+1} - r_i \le (\log N)^{k\delta} \quad (1\le j_2 < s_2)
\]
but
\[
h_{s_1+1} - h_{s_1} > (\log N)^{(k+1)\delta} \quad\mbox{and}\quad
r_{s_2+1} - r_{s_2} > (\log N)^{(k+1)\delta}.
\]
Here we set
\[
a = \sum_{j_1=1}^{s_1} h_{j_1} 2^{k_{s_1}-k_{j_1}} \quad\mbox{and}\quad  b =  \sum_{j_2=1}^{s_2} r_{j_2} 2^{\ell_{s_2}-\ell_{j_2}} 
\]
and use the upper bounds
\[
\sum_{j_1=s_1+1}^{d_1} h_{j_1} 2^{-k_{j_1}-1}  \ll (\log N)^{2\lambda_0} 2^{- k_{s_1}-(\log N)^{(k+1)\delta}} 
\quad\mbox{and}\quad
\sum_{j_2=s_2+1}^{d_2} r_{j_2} 3^{-\ell_{j_2}-1}  \ll (\log N)^{2\lambda_0} 3^{-\ell_{s_2} - (\log N)^{(k+1)\delta}} 
\]
to obtain the lower bound for 
\begin{align*}
|2\alpha| &\ge \left| a 3^K 2^{-k_{s_1}} +  2 b 3^{-\ell_{s_2}-1} \right| - 
\left| 3^K \sum_{j_1=s_1+1}^{d_1} h_{j_1} 2^{-k_{j_1}} \right| - \left| 2 \sum_{j_2=s_2+1}^{d_2} r_{j_2} 3^{-\ell_{j_2}-1}  \right| \\
&\ge \max\left( \left| a 3^K 2^{-k_{s_1}} \right|, \left|  2 b 3^{-\ell_{s_2}-1} \right| \right)
e^{-C' \log\log N (\log N)^{k\delta} } \\
&- O\left(
 \max\left( \left| 3^K 2^{-k_{s_1}} \right|, \left| 2\, 3^{-\ell_{s_2}-1} \right| \right) 
(\log N)^{2\lambda_0} e^{-c (\log N)^{(k+1)\delta}}
      \right) \\
&\gg \max\left( \left| a 3^K 2^{-k_{s_1}} \right|, \left| 2 b 3^{-\ell_{s_2}-1} \right| \right)
e^{-C' \log\log N (\log N)^{k\delta} } \\
&\gg \frac{e^{c'(\log N)^\eta}}{N} \gg \frac{ (\log N)^\lambda}{N}.
\end{align*}
This completes the case $H=0$.

In the case $H\ne 0$ we proceed very similiarly to the case $d_2 = 0$. 
Since $\| 2 \alpha\| \le 1$ we certainly have the upper bound
\[
|H| \le 1+ (\log N)^\lambda \frac{3^K}{2^{k_1}} 
\]
We now reduce the general case to the coprime one and apply Lemma~\ref{LeDio4}.

\medskip

This completes the proof of Lemma~\ref{Leexponentialsums}.

\subsection{Proof of Lemma~\ref{LeKey}}\label{sec:LeKey}

We follow \cite{BassilyKatai1995} and \cite{Drmota2001}. 
Let $f_{b,q,\Delta}(x)$ be defined by
\[
f_{b,q,\Delta}(x) := \frac 1{\Delta} \int_{-\Delta/2}^{\Delta/2} 
{\bf 1}_{[\frac bq,\frac{b+1}q]}(\{x+z\})\, dz,
\]
where ${\bf 1}_A$ denotes the characteristic function of the set $A$
and $\{x\} = x-[x]$ the fractional part of $x$. The Fourier
coefficients of the Fourier 
series $f_{b,q,\Delta}(x) = \sum_{m\in\ZZ} d_{m,b,q,\Delta} e(mx)$
are given by
\[
d_{0,b,q,\Delta} = \frac 1q
\]
and for $m\ne 0$ by 
\[
d_{m,b,q,\Delta} =  \frac{e\left(-\frac {mb}q\right)-
e\left(-\frac {m(b+1)}q\right)}{2\pi i m} \cdot
\frac{e\left(\frac {m\Delta}2\right)-e\left(-\frac {m\Delta}2\right)}
{2\pi i m\Delta}.
\]
Note that $d_{m,b,q,\Delta} = 0$ if $m\ne 0$ and $m\equiv 0\bmod q$ and that
\[
|d_{m,b,q,\Delta}| \le \min\left( \frac 1{\pi |m|},\frac 1{\Delta\pi m^2}\right). 
\]
By definition we have $0\le f_{b,q,\Delta}(x)\le 1$ and
\[
f_{b,q,\Delta}(x) = \left\{ \begin{array}{cl}
1 & \mbox{if } x\in \left[ \frac bq + \Delta, \frac{b+1}q -\Delta\right], \\
0 & \mbox{if } x\in [0,1] \setminus \left[ \frac bq - \Delta, \frac{b+1}q +\Delta\right].
\end{array}\right.
\]
So if we set
\[
t(y_1,y_2) := \prod_{j_1=1}^{d_1} f_{b_{j_1},2,\Delta}
\left( \frac {y_1}{2^{k_{j_1}+1}} \right) \prod_{j_2=1}^{d_2} f_{c_{j_2},3,\Delta}
\left( \frac {y_2}{3^{\ell_{j_2}+1}} \right)
\]
then we get for $\Delta< 1/12$
\begin{eqnarray*}
&&  \left|\# \left\{ n < N,\, n \equiv i \bmod 2
 \,|\, \varepsilon_{2,k_{j_1}}(3^K n) = b_{j_1},\ 1\le j_1\le d_1,\ 
\varepsilon_{3,\ell_{j_2}}(n) = c_{j_2},\ 1\le j_2\le d_2 \right\} 
- \sum_{n<N,\, n\equiv i \bmod 2} t(3^K n, n)\right| \\
&&\quad \le 
\sum_{j_1=1}^{d_1}   \#\left\{n< N\left| \left\{\frac{3^K n}{2^{k_{j_1}+1}} \right\} 
\in U_{b_{j_1},2,\Delta} \right.\right\} 
+ \sum_{j_2=1}^{d_2} 
\#\left\{n< N\left| \left\{\frac{n}{3^{\ell_{j_2}+1}} \right\} 
\in U_{c_{j_2},3,\Delta} \right.\right\} \\
&&\quad \ll \Delta N +  N \sum_{j_1=1}^{d_1} D_1(k_{j_1}) + N \sum_{j_2=1}^{d_2}  D_2(\ell_{j_2})
\end{eqnarray*}
where 
\[
U_{b,q,\Delta} := [0,\Delta]\cup \bigcup_{b=1}^{q-1} \left[ \frac bq-\Delta,\frac bq + \Delta \right]
\cup [1-\Delta,1].
\]
and $D_1(k_{j_1})$ and $D_2(\ell_{j_2})$, respectively, denote the discrepancies of the 
sequences $( 3^K n 2^{-k_{j_1}-1} \bmod 1 : n< N)$ and $( n 3^{-\ell_{j_2}-1} \bmod 1 : n< N)$.
The discrepancies $D_1(k_{j_1})$ and $D_2(\ell_{j_2})$
that can be bounded with the help of the Erd\H os-Turan inequality and exponential sum estimates.

For $D_2(\ell_{j_2})$ we directly obtain 
\begin{align*}
D_2(\ell_{j_2}) &\ll \frac 1H + \sum_{h=1}^H \frac 1h \left| \frac 1N \sum_{n<N} e\left( h  n 3^{-\ell_{j_2}-1} \right) \right| \\
&\ll \frac 1H +  \log H \frac{3^{\ell_{j_2}}}{N} \ll (\log N)^{-\lambda}
\end{align*}
by using the estimate (\ref{eqexpsumtrivial}), setting $H = (\log N)^{\lambda_0}$
(for some $\lambda_0 \ge \lambda$)  
and applying the bound $\ell_{j_2} \le N_2 - N_2^\eta$.

For $D_1(k_{j_1})$ we have to be slightly more careful but we can use the bounds provided in 
Section~\ref{subsectiond2=0} (note that $\lambda$ is replaced by $\lambda+1$):
\[
\sum_{n<N} e\left( h  3^K n 2^{-k_{j_1}-1} \right) \ll  N (\log N)^{-\lambda-1}
\]
Obviously this leads to 
\begin{align*}
D_1(k_{j_1}) &\ll \frac 1H + \sum_{h=1}^H \frac 1h \left| \frac 1N \sum_{n<N} e\left( h 3^K n 2^{-k_{j_1}-1} \right) \right| \\
&\ll \frac 1H +  \log H (\log N)^{-\lambda-1} \ll (\log N)^{-\lambda}.
\end{align*}

Summing up, by setting $\Delta = (\log N)^{-\lambda_0}$ for some $\lambda_0 \ge \lambda$ we get
\begin{eqnarray*}
&&  \left|\# \left\{ n < N \,|\, \varepsilon_{2,k_{j_1}}(3^K n) = b_{j_1},\ 1\le j_1\le d_1,\ 
\varepsilon_{3,\ell_{j_2}}(n) = c_{j_2},\ 1\le j_2\le d_2 \right\} 
- \sum_{n<N} t(3^K n, n)\right| \\
&&\quad \ll N (\log N)^{-\lambda}.
\end{eqnarray*}

For convenience, let 
${\bf h} = (h_1, \ldots, h_{d_1})$ and ${\bf r}= (r_1, \ldots, r_{d_2})$
denote $d_1$- and $d_2$-dimensional integer vectors and
${\bf v} = \left( 2^{-k_1-1}, \ldots, 2^{-k_{d_1}-1}  \right)$,
${\bf w} = \left( 3^{-\ell_1-1}, \ldots, 3^{-\ell_{d_2}-1}  \right)$.
Furthermore set
\[
T_{{\bf h}, {\bf r}} := \prod_{j_1=1}^{d_1} d_{h_{j_1},b_{j_1},2,\Delta}
\prod_{j_2=1}^{d_2} d_{r_{j_2},c_{j_2},3,\Delta}.
\]
Then $t(y_1,y_2)$ has the Fourier series expansion
\[
t(y_1,y_2) = \sum_{{\bf h}, {\bf r}} T_{{\bf h}, {\bf r}}
e\left( {\bf h}\cdot {\bf v} \,  y_1 +  {\bf r}\cdot {\bf w} \, y_2 \right).
\]
Thus, we are led to consider the  sums
\begin{equation}\label{eqS1}
\sum_{n< N,\, n\equiv i \bmod 2} t(3^K n, n) 
=\sum_{{\bf h}, {\bf r}} T_{{\bf h}, {\bf r}}
\sum_{n<N } e\left( \left( 3^K {\bf h}\cdot {\bf v}  +  {\bf r}\cdot {\bf w} \right) n \right)
\end{equation}
If ${\bf h} = {\bf r} = {\bf 0}$ then
\[
T_{ {\bf 0}, {\bf 0}}
\sum_{n<N,\, n\equiv i \bmod 2} e\left( 0 \right)
= \frac {N/2 + O(1)}{2^{d_1} 3^{d_2}}
\]
which provides the leading term. Furthermore, 
we have (for $\Delta = (\log N)^{\lambda_0}$) the estimate
\[
\sum_{({\bf h},{\bf r})  \ne ({\bf 0},{\bf 0}) }  | T_{{\bf h}, {\bf r}}| \ll 
(2 + 2 \log (1/\Delta))^{d_1+d_2} \ll (\log\log N)^{d_1+d_2}
\]
and
\[
\sum_{  \|({\bf h},{\bf r})\| \ge (\log N)^{2\lambda_0} }  | T_{{\bf h}, {\bf r}}|  \ll (\log N)^{-\lambda_0}.
\]
Thus, 
\[
\sum_{n< N,\, n\equiv i \bmod 2} t(3^K n, n) = \frac {N/2 + O(1)}{2^{d_1} 3^{d_2}} 
+ O\left(  (\log\log N)^{d_1+d_2} (\log N)^{-\lambda} \right) 
+ O\left( (\log N)^{-\lambda_0}   \right)
\]
This completes the proof of Lemma~\ref{LeKey}.

\subsection{Completion of the proof of Proposition~\ref{Pro2}}\label{sec:proofPro2}

We finally show that Lemma~\ref{LeKey} implies Proposition~\ref{Pro2} (again we 
follows \cite{Drmota2001}).

The idea is to compare the distribution of $(s_2(3^K n),s_3(n))$, $n< N$, 
$n\equiv i\bmod 2$,  with the distribution
of independent pairs of sums of iid random variables. Let $Z_{2,j}$ be iid random variables 
that are uniformly distribution on $\{0,1\}$ and $Z_{3,j}$ iid random variables on $\{0,1,2\}$
that are also independent of $Z_{2,j}$. Then we consider the pair of random variables
\[
S_2(3^K N) = \sum_{j=0}^{N_1}  Z_{2,j} \quad \mbox{and}\quad  S_3(N) = \sum_{j=0}^{N_2}  Z_{3,j}
\]
and also the trucated versions
\[
\tilde S_2(3^K N) = \sum_{N_1^\eta \le j \le N_1-N_1^\eta}  Z_{2,j} \quad \mbox{and}\quad  
\tilde S_3(N) = \sum_{N_2^\eta \le j \le N_2-N_2^\eta}  Z_{3,j},
\]
Recall that $N_1 = \lfloor \log_{2 } (3^K N) \rfloor$ and $N_2 =  \lfloor \log_{3 }N \rfloor$
and that $0 < \eta < \frac 12$. 
We also set 
\[
\tilde N_1 = |\{j : N_1^\eta \le j \le N_1-N_1^\eta\}| = N_1 - 2 N_1^\eta + O(1)
\]
and
\[
\tilde N_2 = |\{j : N_2^\eta \le j \le N_2-N_2^\eta \}| = N_2 - 2 N_2^\eta + O(1).
\]

Clearly 
\[
\left| S_2(3^K N) - \tilde S_2(3^K N) \right| \ll (\log N)^\eta \quad \mbox{and}\quad 
\left| S_3(N) - \tilde S_3(K N) \right| \ll (\log N)^\eta.
\]
The normalized versions
\[
(Y_{1,3^K N},Y_{2,N}) =  \left( 
\frac{S_2(3^K N) - \frac 12 N_1}{\sqrt{\frac 14 \log_2(3^K N)}}, 
\frac{S_3(N) - N_2}{\sqrt{\frac 23 \log_3( N)}}
\right) 
\]
and
\[
(\tilde Y_{1,3^K N}, \tilde Y_{2,N}) = \left( 
\frac{\tilde S_2(3^K N) - \frac 12 \tilde N_1}{\sqrt{\frac 14 \log_2(3^K N)}}, 
\frac{\tilde S_3(N) - \tilde N_2}{\sqrt{\frac 23 \log_3( N)}}
\right) 
\]
converge then to the two-dimensional normal distribution $N({\bf 0}, {\bf I})$, where
${\bf I}$ denotes the identity matrix. In particular the characteristic functions converge:
\[
\lim_{N\to\infty} \mathbb{E} \left( e^{it_1 Y_{1,3^K N} + it_2 Y_{2,N} } \right) \to e^{-t_1^2/2- t_2^2/2}
\]
and 
\[
\lim_{N\to\infty} \mathbb{E} \left( e^{it_1 \tilde Y_{1,3^K N} + it_2 \tilde Y_{2,N} } \right) \to e^{-t_1^2/2- t_2^2/2}.
\]
Note that the convergence is uniform for $0\le K \le c_2 \log N$. 
Furthermore, we have convergence of all (joint) moments:
\begin{equation}\label{eqmomentsrel}
\lim_{N\to\infty} \mathbb{E}\left[  \tilde Y_{1,3^K N}^{d_1}  \tilde Y_{2,N}^{d_2}  \right]
\to \mu_{d_1} \mu_{d_2}, 
\end{equation}
where $\mu_d = (d-1)!!$ for even $d$ and $\mu_d = 0$ for odd $d$ denote the
moments of the standard normal distribution. Again the convergence is
uniform for $0\le K \le c_2 \log N$. These are standard exercises for sums of
independent random variables.

Next let $D_{2,j, K, N}$ denote the (random) $j$-th digit in the binary expansion of $3^K n$ 
and $D_{3,j,N}$ the (random) $j$-th digit in the ternary expansion of $n$ (again) if $n< N$ with
$n\equiv i \bmod 2$ is chosen uniformly at random.
Then the sum-of-digits functions 
\[
s_2(3^K n) = \sum_{j = 0}^{N_1} D_{2,j, K, N} \quad\mbox{and}\quad
s_3(n) = \sum_{j= 0}^{N_2} D_{3,j, N}
\]
model the random pair $(s_2(3^K N), s_3(n))$ if $n< N$ with $n\equiv i \bmod 2$ 
 is chosen uniformly at random. Again we also consider the 
truncated versions
\[
\tilde s_2(3^K n) = \sum_{N_1^\eta \le j\le N_1 - N_1^\eta} D_{2,j, K, N} \quad\mbox{and}\quad
\tilde s_3(n) = \sum_{N_2 \le j\le N_2 - N_2^\eta} D_{3,j, N}.
\]
and the normalized versions:
\[
(X_{1,3^K N},X_{2,N}) =  \left( 
\frac{s_2(3^K N) - \frac 12 N_1}{\sqrt{\frac 14 \log_2(3^K N)}}, 
\frac{s_3(N) - N_2}{\sqrt{\frac 23 \log_3( N)}}
\right) 
\]
and
\[
(\tilde X_{1,3^K N}, \tilde X_{2,N}) = \left( 
\frac{\tilde s_2(3^K N) - \frac 12 \tilde N_1}{\sqrt{\frac 14 \log_2(3^K N)}}, 
\frac{\tilde s_3(N) - \tilde N_2}{\sqrt{\frac 23 \log_3( N)}}
\right).
\]

\begin{lemma}\label{Lecomparemoments}
For every pair of non-negative integers $d_1,d_2$ and $i\in \{0,1\}$ we have
\[
\lim_{N\to\infty} 
\frac 2N 
\sum_{n< N,\, n\equiv i \bmod 2}
\left( \frac{\tilde s_2(3^K n) - \frac 12 \tilde N_1}
{\sqrt{\frac 14 \log_2(3^K N)} }  \right)^{d_1}
\left( \frac{\tilde s_3(n) - \tilde N_2}
{\sqrt{\frac 23 \log_3( N)} }  \right)^{d_2} = \mu_{d_1} \mu_{d_2}
\]
uniformly for $0\le K \le c_2 \log N$.
\end{lemma}

\begin{proof}
We rewrite the sum-of-digits function 
\begin{align*}
\tilde s_2(3^K n) &= \sum_{N_1^\eta \le j\le N_1 - N_1^\eta} \varepsilon_{2,j}(3^K n) 
= \sum_{N_1^\eta \le j\le N_1 - N_1^\eta} D_{2,j, K, N}, \\
\tilde s_3(n) &= \sum_{N_2^\eta \le j\le N_2 - N_2^\eta} \varepsilon_{3,j}(3^K n) 
= \sum_{N_2^\eta \le j\le N_2 - N_2^\eta} D_{3,j, N}
\end{align*}
and expand the moments
\[
\frac 2N 
\sum_{n< N,\, n\equiv i \bmod 2 }
\left( \tilde s_2(3^K n) - \frac 12 \tilde N_1 \right)^{d_1}
\left( \tilde s_3(n) - \tilde N_2 \right)^{d_2} 
\]
in $\tilde N_1^{d_1} \tilde N_2^{d_2} = O((\log N)^{d_1+d_2}))$ terms of 
\[
\frac 2N \# \left\{ n < N,\, n\equiv i \bmod 2 \,:\, \varepsilon_{2,k_{j_1}}(3^K n) = b_{j_1},\,
1\le j_1\le d_1',\, \varepsilon_{3,\ell_{j_2}}(n) = c_{j_2},\, 1\le j_2 \le d_2'  \right\}.
\]
with $0\le d_1' \le d_1$, $0\le d_2' \le d_2$
and powers of $\tilde N_1$ and $\tilde N_2$. By Lemma~\ref{LeKey} we can replace
these numbers by $2^{-d_1'}3^{-d_2'} + O((\log N)^{-\lambda})$
uniformly for $0\le K\le c_2 \log N$, where we choose
$\lambda > 2(d_1+d_2)$. Clearly the resulting sum equals 
\[
\mathbb{E}\left[ \left( \tilde S_2(3^K N) - \frac 12 \tilde N_1 \right)^{d_1}
\left( \tilde S_3(N) - \tilde N_2 \right)^{d_2} \right] 
+ O\left( (\log N)^{2(d_1+d_2)-\lambda} \right).
\]
Finally by dividing the resulting equation by
\[
\left( \frac 14 \log_2(3^K N) \right)^{\frac{d_1}2} 
\left( \frac 23 \log_3( N) \right)^{\frac{d_2}2}
\]
and by using the relation (\ref{eqmomentsrel}) we complete the proof of the lemma.
\end{proof}

Lemma~\ref{Lecomparemoments} directly implies that 
the truncated and normalized pair of random variables
\[
(\tilde X_{1,3^KN}, \tilde X_{2,N})
\]
converges weakly to the $2$-dimensional normal distribution $N({\bf 0}, {\bf I})$. 
Since $\eta < \frac 12$ the same holds for
the untruncated pair
\[
(X_{1,3^KN}, X_{2,N}).
\]
Hence, we also have for the characteristic function
\begin{equation}\label{eqfinalrelation}
\lim_{N\to\infty} \mathbb{E} \left( e^{it_1 X_{1,3^K N} + it_2 X_{2,N} } \right) 
\to e^{-t_1^2/2- t_2^2/2}.
\end{equation}
More precisely by using the Taylor expansion for $e^{it}$ convergence of moments
and in particular the uniformity for $0\le K\le c_2 \log N$ can be directly 
transformed in uniform convergence for the characteristic function. 
Thus, (\ref{eqfinalrelation}) holds uniformly for $0\le K\le c_2 \log N$

By rewriting (\ref{eqfinalrelation}) in terms of the sum-of-digits functions this
is precisely (the first part of) Proposition~\ref{Pro2}. 

As mentioned already several times, the second part of Proposition~\ref{Pro2}
follows very similarly. This completes the proof of Proposition~\ref{Pro2}.

\bibliographystyle{siam}
\bibliography{count_collisions}

\end{document}